\documentclass{aptpub}

\usepackage[a4paper]{geometry}
\usepackage{verbatim}
\usepackage{graphicx}
\usepackage{amsmath,amsfonts}
\usepackage[round]{natbib}
\usepackage{theorem}

\DeclareMathOperator{\erfc}{erfc}

\def\c{\mathrm{c}}
\def\textfrac#1#2{{\textstyle\frac{{#1}}{{#2}}}}
\def\Levy{L\'evy}

\makeatletter
\AtBeginDocument{\let\@biblabel\NAT@biblabelnum}
\makeatother

\authornames{F. HUBALEK AND A. E. KYPRIANOU} 
\shorttitle{} 

\begin{document}

\title{Old and new examples of scale functions\\ for
spectrally negative L\'evy processes} 

\bigskip
\authorone[Vienna University of Technology]{F. Hubalek} 
\addresstwo{Vienna University of Technology,
Financial and Actuarial Mathematics,
Wiedner Hauptstra\ss{}e~8/105--1,
A-1040 Vienna, Austria. email: fhubalek@fam.tuwien.ac.at }
\authortwo[The University of Bath]{A. E. Kyprianou} 
\addressone{Department of Mathematical Sciences,
The University of Bath,
Claverton Down,
Bath BA2 7AY.
email: a.kyprianou@bath.ac.uk}

\begin{abstract}
We give a review of the state of the art with regard to the theory of scale functions for spectrally 
negative \Levy{} processes. From this we introduce a general method for generating new families of scale functions. Using this method we introduce a new family of scale functions belonging to the {\it Gaussian Tempered Stable Convolution} (GTSC) class.  We give particular emphasis to special
cases as well as cross-referencing their analytical behaviour against known general considerations.
\end{abstract}



\section{Spectrally negative \Levy{} processes and scale functions}
Let $X=\{X_t : t\geq 0\}$ be a \Levy{} process defined on a
filtered probability space $(\Omega , \mathcal{F}, \mathbb{F}, \mathbb{P})$,
where $\{\mathcal{F}_t: t\geq 0\}$ is the filtration generated by $X$ satisfying the usual conditions.
For  $x\in \mathbb{R}$ denote by  $\mathbb{P}_{x}$ the law of $X$ when it is started at $x$ and
write simply  $\mathbb{P}_{0}=\mathbb{P}$. 
Accordingly we shall write $\mathbb{E}_x$ and $\mathbb{E}$ for the associated expectation operators. In this paper we shall
assume throughout that $X$ is {\it spectrally negative} meaning here that it has no positive jumps and that it is not the
negative of a subordinator. It is well known that the latter allows us to talk about the Laplace exponent $\psi(\theta):=
\log\mathbb{E}[e^{\theta X_1}]$ for $\Re(\theta)\geq 0$ where in particular we have the \Levy{}-Khintchine representation

\begin{equation}
\psi(\theta) = -a\theta + \frac{1}{2}\sigma^2\theta^2 + \int_{(-\infty, 0)}(e^{\theta x} -1- x\theta\mathbf{1}_{\{x>-1\}})\Pi(d x)
\label{Laplace}
\end{equation}
where $a\in\mathbb{R}$, $\sigma\geq 0$ is the Gaussian coefficient and $\Pi$ is a measure concentrated on $(-\infty,0)$ satisfying $\int_{(-\infty,0)}(1\wedge x^2) \Pi(dx)<\infty$. 
The, so-called, \Levy{} triple $(a,\sigma, \Pi)$ completely characterises the process $X$.

For later reference we also introduce the 
function $\Phi: [0,\infty)\rightarrow [0,\infty)$ as 
the right inverse of $\psi$  on $(0,\infty)$ so that 
for all $q\geq 0$
\begin{equation}
\Phi(q) =  \sup\{\theta\geq 0 : \psi(\theta) = q\}.
\end{equation}
Note that it is straightforward to show that  $\psi$ is a strictly convex
function which is zero at the origin and tends to infinity at infinity and hence
there are at most two solutions of the equation $\psi(\theta)=q$.

\bigskip

Suppose now we define the stopping times for each $x\in\mathbb{R}$
\begin{equation}
\tau^+_x =\inf\{t> 0 : X_t >x\}\text{ and }\tau^-_x = \inf\{t> 0 : X_t < x\}.
\end{equation}
A fluctuation identity with a long history concerns the probability that $X$ exits an interval $[0,a]$ (where $a>0$) into
$(a,\infty)$ before exiting into $(-\infty,0)$ when issued at $x\in[0,a]$.  In particular it is known that 
\begin{equation}
\mathbb{E}_x(e^{-q\tau^+_a} \mathbf{1}_{\{\tau^+_a< \tau^-_0\}}) = \frac{W^{(q)}(x)}{W^{(q)} (a)}
\label{twosidedexit}
\end{equation}
where $x\in (-\infty,a]$, $q\geq 0$ and the function $W^{(q)} : \mathbb{R}\rightarrow [0,\infty)$ is defined up to a
multiplicative constant as follows. On $(-\infty,0)$ we have $W^{(q)}(x)=0$ and otherwise $W^{(q)}$ is a continuous function
(right continuous at the origin) with Laplace transform 
\begin{equation}
\int_0^\infty e^{-\theta x}W^{(q)}(x)d x = \frac{1}{\psi(\theta) -q} \text{ for }\Re(\theta) >\Phi(q).
\label{LT}
\end{equation}

\bigskip

The functions $\{W^{(q)} : q\geq 0\}$ are known as {\it scale functions} and for convenience and consistency with existing literature we write $W$ in place of $W^{(0)}$.
Identity (\ref{twosidedexit}) exemplifies the relation between scale functions for $q=0$ and the classical ruin problem. Indeed setting $q=0$ we have that $\mathbb{P}_x(\tau^-_0<\tau^+_a) = 1-W(x)/W(a)$ and so assuming that $\psi'(0+)>0$ and taking limits as $a\uparrow\infty$ it is possible to deduce that $W(\infty)^{-1} = \psi'(0+)$ and hence 
\begin{equation}
\mathbb{P}_x(\tau^-_0<\infty)=1 - \psi'(0+)W(x).
\end{equation}
It is in this context of ruin theory that scale functions
make their earliest appearance in the works of \citet{Zol1964}, 
\citet{Tak1966} and then later either explicitly or implicitly in the work of \citet{Eme1973} 
and \citet{Kor1974,Kor1975}, \citet{Sup1976} and \citet{Rog1990}.
The real value of scale functions as a  class with which one may express a whole range of fluctuation 
identities for spectrally negative \Levy{} processes became apparent in the work of \cite{Chau94, Chau96} 
of \citet{Ber1996, Ber1997} and an ensemble of subsequent articles; 
see for example \citet{Lam2000}, \citet{AKP2004}, 
\citet{Pis2003,Pis2004,Pis2005,Pis2006}, 
\citet{KP2006}, \cite{DK2006} and \citet{Don1991,Don2005,Don2007}. 
Moreover with the advent of these new fluctuation identities and a better understanding of the analytical 
properties of the function $W^{(q)}$ came the possibility of revisiting and solving a number of classical 
and modern problems from applied probability, but now with the underlying source of randomness being a 
general spectrally negative \Levy{} processes. For example, in the theory of optimal 
stopping \citet{ACU2002, AKP2004, AK2004} and \citet{Kyp2006}, in the theory of optimal 
control \citet{APP2007} and \citet{Ronnie2007}, in the theory of queuing and storage 
models \citet{DGM2004} and \citet{BBK2007}, in the theory of branching 
processes \citet{Bin1976} and \cite{Lam2007}, in the theory of insurance risk and 
ruin \citet{CY2005}, \citet{KKM2004}, 
\citet{KK2006}, \citet{DK2006}, 
\cite{KP2007},  \citet{RZ2007}, \cite{Ronnie2007} and \citet{KRS2008}, in the theory of credit risk 
\citet{HR2002} and \citet{KS2007} and in the theory of fragmentation \cite{krell2007}. 

\bigskip

Although scale functions are now firmly embedded within the theory of spectrally negative \Levy{} 
processes and their applications, and although there is a reasonable understanding of how they 
behave analytically (see for example the summary in \citet{Kyp2006}), one of their main failings 
from a practical point of view is that there are limited number of examples which are available `off-the-shelf' for modeling purposes.

The original purpose of this paper was to introduce a previously unknown family of scale functions. However, in doing so, we
came across a number of additional examples which were seemingly less well known in the literature or are themselves new and
concurrent with this article. We shall therefore initially spend some time documenting known examples of scale
functions.
It seemed then sensible to also document known general analytical properties of scale functions and this we have also done in Section \ref{known-analytical}.

Following that we shall expose
a new family of explicit examples of scale functions as well specifying in detail the associated \Levy{}
processes. By `explicit' we mean that, as in the below, known examples, we are able to give exact analytical
expressions for the scale functions. By `specifying in detail' we mean that we shall describe the \Levy{} triple
$(a,\sigma,\Pi)$ of the associated spectrally negative \Levy{} process and in particular it will turn out that  a density for the \Levy{} measure can be obtained and we are able to identify it as a known function. Moreover,
it turns out that the associated \Levy{} processes discussed may also be described as the
independent sums of other, well known \Levy{} processes.


\section{Known examples of scale functions\label{known}}
At the time of writing, and to the best of the authors' knowledge, there are 
essentially six main examples of explicit classes; three of which deal with the case of a compound Poisson process with negative jumps and  positive drift. 
\begin{description}
\item[1.] Brownian motion with drift.
The associated Laplace exponent is written $\psi(\theta) = \sigma^2\theta^2/2 +\mu\theta$, where $\sigma>0$ and $\mu\in\mathbb{R}$ then, for $x\geq 0$,
\begin{equation}
W^{(q)}(x) = \frac{2}{\sqrt{2q\sigma^2 + \mu}}e^{-\mu x/\sigma^2}\sinh(\frac{x}{\sigma^2 }\sqrt{2q\sigma^2 + \mu})
\end{equation}
for $q\geq 0$, where in the case that $q=\mu=0$ the above expression is to be taken in the limiting sense.

\item[2.] Spectrally negative stable processes with stability parameter $\beta\in (1,2)$.
When $\psi(\theta)=\theta^\beta$ with $\beta\in(1,2)$, we have for $x\geq 0$
\begin{equation}
W^{(q)}(x) = \beta x^{\beta -1} E'_{\beta,1}(qx^\beta)
\label{stablescale}
\end{equation}
for $q\geq 0$.
where $E_{\beta,1}(z) =\sum_{k\geq 0} z^k /\Gamma(1+\beta k)$ is the Mittag-Leffler function.

There also exists  an expression for the scale function of the aforementioned $\beta$-stable process but now with a strictly positive drift (but only for the case $q=0$) implicitly in the paper of \cite{Furrer1998}.
Indeed, by considering the expression there for the ruin probability one may extract the following, 
\begin{equation}
W(x) = \frac{1}{c}\left(1-E_{\beta-1}(-c x^{\beta-1})\right)
\end{equation}
for $x\geq 0$.
Note that by taking $c\downarrow 0$ one obtains an expression which agrees with (\ref{stablescale}), namely $x^{\beta-1}/\Gamma(\beta)$.

\item[3.] 
Generally speaking it is possible to deal with a strictly positive linear drift minus a compound Poisson process process
with positive jumps with a rational Laplace transform.  The calculations are however rather extensive. See for example
\cite{Mor} where the scale function can be extracted from their expression for the ruin probability. In the special case
that the drift is 
denoted $\c$ and jumps arrive at rate $\lambda$ and are exponentially distributed with 
parameter $\mu$ such that $\c -\lambda/\mu >0$, that is to say $\psi(\theta) =  \c\theta - \lambda(1 - \mu(\mu+\theta)^{-1})$, we have for the case $q=0$  and $x\geq 0$,
\begin{equation}
W(x)  = \frac{1}{\c}\left(1 + \frac{\lambda }{\c \mu - \lambda}(1 - e^{(\mu - \c^{-1}\lambda)x})\right).
\end{equation}

The expression for the scale function in the case that $q>0$ with exponentially distributed jumps, and indeed for the
slightly more general case where an independent Gaussian component is added in, can in principle be extracted from
\cite{KouWang}. 

\item[4.] This example is taken from \cite{Asm2000} and again comes from the case of a known ruin probability giving a scale function for the case $q=0$ and positive drift. Consider a spectrally negative compound Poisson process whose jumps are exactly of size $\alpha\in(0,\infty)$, 
with arrival rate $\lambda >0$ and with positive drift $c>0$ such that $c-\lambda\alpha>0$. In that case, $\psi(\theta) = c\theta - \lambda(1-e^{-\alpha\theta})$ and for $x\geq 0$ we have

\begin{equation}
W(x) = \frac{1}{c}\sum_{n=1}^{\lfloor x/\alpha\rfloor} e^{-\lambda(\alpha n - x)/c}\frac{1}{n!}\left(\frac{\lambda}{c}\right)^n (\alpha n - x)^n
\end{equation}
where $\lfloor x/\alpha\rfloor$ is the integer part of $x/\alpha$.

\item[5.]  The next example comes from \cite{Abate-Whitt} who themselves  refer to older work of \cite{Boxma-Cohen}.

Consider a unit-rate linear drift minus a compound Poisson process of rate
$\lambda>0$ with jumps whose distribution $F$ has Laplace transform given by 
\begin{equation}
\int_0^\infty e^{-\theta x}F(dx)=
1 - \frac{\theta}{(\mu+\sqrt{\theta})(1+\sqrt{\theta})}
\end{equation}
for $\Re(\theta)\geq 0$
The latter Laplace transform corresponds to a random variable whose mean is equal to $\mu^{-1}>0$ and whose tail distribution takes the form
\begin{equation}
F(x,\infty) = (2x+1)\eta(x)- 2\sqrt{\frac{x}{\pi}}, \, x\geq 0
\end{equation}
when $\mu=1$ and otherwise when $\mu\neq 1$
\begin{equation}
F(x,\infty) = \left(\frac{1}{1-\mu}\right)(\eta(x) - \mu\eta(x\mu^2)), \, x\geq 0
\end{equation}
where $\eta(x)=e^x\text{erfc}(\sqrt{x})$. 
It is assumed that the underlying L\'evy process drifts to $\infty$. Since $\mathbb{E}(X_1)=1-\lambda/\mu$ the latter assumption is  tantamount to $\lambda/\mu<1$.

Said another way, we are interested in a spectrally negative L\'evy process with Laplace exponent given by 
\begin{equation}
\psi(\theta) = \theta - \frac{\lambda\theta}{(\mu+\sqrt{\theta})(1+\sqrt{\theta})}.
\end{equation}
For the scale function with $q=0$ it is known that for $x\geq 0$
\begin{equation}
W(x) =\frac{1}{1-\lambda/\mu}\left(1 - \frac{\lambda/\mu}{\nu_1 - \nu_2}(\nu_1 \eta(x\nu_2^2) - \nu_2\eta(x\nu_1^2))\right).
\end{equation}
and 
\begin{equation}
\nu_{1,2} = \frac{1+\mu}{2}\pm\sqrt{\left(\frac{1+\mu}{2}\right)^2 - \left(1-\frac{\lambda}{\mu}\right)\mu}.
\end{equation}

\item[6.] The final two examples, both scale functions only for the case $q=0$, appeared very recently in the theory of positive self-similar Markov processes, see \cite{CKP}. See also \cite{CC06} for the  origin of the underling L\'evy processes. The first example is the scale function which belongs to a spectrally  negative L\'evy process with no Gaussian component, 
whose \Levy{} measure takes the form
\begin{equation}
\Pi(dy)=\frac{e^{(\beta-1)y}}{(e^y -1)^{\beta +1}}dy, \, y<0
\end{equation}
where $\beta\in(1,2)$ and whose Laplace exponent takes the form
\begin{equation}
\psi(\theta) = \frac{\Gamma(\theta-1+\beta)}{\Gamma(\theta-1)\Gamma(\beta)}
\end{equation}
for $\Re(\theta)\geq 0$.
Note that $\psi'(0+)<0$ and hence the process drifts to $-\infty$.
In that case  it was found that for $x\geq 0$
\begin{equation}
W(x) = (1-e^{-x})^{\beta-1}e^x.
\end{equation}

The second example is the scale function associated with the aforementioned
L\'evy process when conditioned to drift to $\infty$. It follows that there is
still no Gaussian component  and the \Levy{} measure takes the form 
\begin{equation}
\Pi(dy)=\frac{e^{\beta y}}{(e^y -1)^{\beta +1}}d y, \, y<0
\end{equation}
and the associated Laplace exponent is given by 
\[
\psi(\theta) = \frac{\Gamma(\theta+\beta)}{\Gamma(\theta)\Gamma(\beta)}
\]
for $\Re(\theta)\geq 0$.
The scale function is then given for $x\geq 0$ by
\begin{equation}
W(x) = (1-e^{-x })^{\beta-1}.
\end{equation}
\end{description}

\section{Known analytic properties of scale functions.}\label{known-analytical}
Although at the time of writing further examples other than those above are lacking, there are a collection of general properties known for scale  functions,
mostly concerning their behaviour at $0$ and $\infty$. For later reference in this text  and to give credibility to
some we review them briefly here. As usual,  $(a, \sigma, \Pi)$ denotes  the \Levy{} triple of a general spectrally
negative \Levy{} process. 
\subsubsection*{Smoothness} 
The following facts are taken from \citet{Lam2000}, \citet{CK2007}, \citet{KRS2008} and  \citet{doneySdP}. It is known that if $X$ has paths of bounded variation then,  for all $q\geq 0$,
$W^{(q)}|_{(0,\infty)}\in C^1(0,\infty)$ if and only if $\Pi$ has no atoms.  
In the case that $X$ has paths of  unbounded variation, it is known that, for all $q\geq 0$, $W^{(q)}|_{(0,\infty)}\in
C^1(0,\infty)$.  Moreover if $\sigma > 0$ then $C^1(0,\infty)$ may be replaced by $C^2(0,\infty)$. It was also
noted  by Renming Song (see the remarks in \citet{CK2007}) that if $\Pi(-\infty, -x)$ is completely monotone  then
$W^{(q)}|_{(0,\infty)}\in C^\infty(0,\infty)$. 
\subsubsection*{Concavity and convexity} 
It was shown in \cite{Ronnie2007} that the latter assumption 
that $\Pi(-\infty, -x)$ is completely monotone also implies that $W^{(q)\prime}(x)$ is convex for $q>0$. Note in particular, 
the latter implies that there exists an $a^*\geq 0$ such that $W^{(q)}$ is 
concave on $(0,a^*)$ and convex on $(a^*,\infty)$. In the case that $\psi'(0+)\geq 0$ and $q=0$ the argument in \cite{CK2007} shows that $a^*=\infty$ and $W$ is necessarily concave.
In \cite{KRS2008} it  is shown that if $\Pi(-\infty, -x)$ has a density on $(0,\infty)$ which is non-increasing and log-convex then for each $q\geq 0$, the scale function $W^{(q)}(x)$ and its first derivative are convex beyond some finite value of $x$.

\subsubsection*{Continuity at the origin} For all $q\geq 0$, 
\begin{equation} 
W^{(q)}(0+) =  \left\{ 
\begin{array}{ll} 0 & \mbox{if $\sigma>0$ or $\int_{(-1,0)}(-x)\Pi(d x)=\infty$} \\
\c^{-1} & \mbox{if $\sigma=0$ and $\int_{(-1,0)}(-x)\Pi(d x)<\infty$},  
\end{array} \right. 
\label{W(0)}
\end{equation} 
where $\c = -a - \int_{(-1,0)} x \Pi(d x)$.
\subsubsection*{Derivative at the origin} 
For all $q\geq 0$, 
\begin{equation} 
W^{(q)\prime}(0+) = \left\{ 
\begin{array}{ll}
2/\sigma^2 & \text{ if } \sigma >0\\ 
\infty & \text{ if }\sigma=0 \text{ and } \Pi(-\infty,0)=\infty\\ 
(q +\Pi(-\infty, 0))/\c^2 & \text{ if }\sigma=0 \text{ and }\Pi(-\infty, 0)<\infty. 
\end{array} \right. 
\label{W'(0)}
\end{equation}
\subsubsection*{Behaviour at $\infty$ for $q=0$} 
As $x\uparrow\infty$ we have 
\begin{equation}\label{W(infty)} 
W(x)\sim \left\{ 
\begin{array}{ll}
1/\psi'(0+) & \text{ if } \psi'(0+) > 0 \\ 
e^{\Phi(0)x}/\psi'(\Phi(0)) & \text{ if }\psi'(0+)<0. 
\end{array} \right. 
\end{equation} 
When $\mathbb{E}(X_1)=0$  a number of different asymptotic behaviours may
occur.  For example, if writing $\psi(\theta)=\theta\phi(\theta)$ (which is possible by the Wiener-Hopf factorization), we have $\phi'(0+)<\infty$ then
 $W(x)\sim x/\phi'(0+) $ as $x\uparrow\infty$. 
\subsubsection*{Behaviour at $\infty$ for $q>0$} 
As $x\uparrow\infty$ we have 
\begin{equation}\label{Wq(infty)}
W^{(q)}(x)\sim e^{\Phi(q)x}/\psi'(\Phi(q)) 
\end{equation} 
and thus there is asymptotic exponential growth.
\section{Methodology for new examples}
For any given spectrally negative \Levy{} process, the scale functions are intimately connected to 
descending ladder height process and this forms the key to  constructing new examples. 
Therefore we shall briefly review the connection between scale functions and the descending ladder 
height process before describing the common methodology that leads to the new examples of scale functions.
For a more detailed account of this connection, the reader is referred to the books of \citet{Ber1996},  \citet{Kyp2006} or \cite{Don2007}.

It is straightforward to show that the process $X-\underline{X}:=\{X_t-\underline{X}_t:t\geq 0\}$, 
where $\underline{X}_t := \inf_{s\leq t}X_s$, is a strong Markov process with state space $[0,\infty)$. 
Following standard theory of Markov local times (cf. Chapter IV of \cite{Ber1996}), it is possible to construct a  local time at zero 
for $X-\underline{X}$ which we henceforth refer to as $L=\{L_t : t\geq 0\}$. Its inverse process, $L^{-1}:=\{L^{-1}_t : t\geq 0\}$ where $L^{-1}_t =\inf\{s>0 : L_s >t\}$, 
is a (possibly killed) subordinator.
Sampling $X$ at $L^{-1}$ we recover the points of minima of $X$. If we define $H_t =X_{L^{-1}_t}$ 
when $L^{-1}_t<\infty$ with $H_t  =\infty$ otherwise, then it is known that the 
process $H=\{H_t : t\geq 0\}$ is a (possibly killed) subordinator. The latter is known as 
the {\em descending ladder height process}. Moreover,  if $\Upsilon$ is 
the \Levy{} measure of $H$ then 
\begin{equation}
\Upsilon(x,\infty) = e^{\Phi(0)x}\int_x^\infty e^{-\Phi(0)u}\Pi(-\infty,-u)d u\text{ for }x>0,
\end{equation}
see for example \citet{Vig2002}.
Further, the subordinator has a drift component if and only if $\sigma>0$ in which case the 
drift is necessarily equal to $\sigma^2/2$. The killing rate of $H$ is given by the constant
$\mathbb{E}(X_1)\vee 0$.

The starting point for the relationship between the descending ladder height process and scale functions 
is given by the Wiener-Hopf factorization. In `Laplace form' for spectrally negative L\'evy processes this takes the appearance
\begin{equation}
\psi(\theta) = (\theta -\Phi(0))\phi(\theta)
\label{WHfact}
\end{equation}
where 
\begin{equation}\label{phipsi}
\phi(\theta) = -\log \mathbb{E}(e^{-\theta H_1}),\qquad
\psi(\theta)=\log\mathbb{E}(e^{\theta X_1}),
\end{equation}
and $\Re(\theta)\geq 0$.
With this form of the Wiener-Hopf factorization in mind, we appeal principally to two techniques.
\begin{enumerate}
\item We choose the process $H$, or equivalently $\phi(\theta)$, so that the Laplace inversion of
(\ref{LT}) may be performed. In some cases, when the process does not drift to $-\infty$ (or equivalently $\psi'(0+)\geq 0$), it can be worked to ones advantage that, after an integration by parts, one also has that
\begin{equation}
\int_0^\infty e^{- \theta x}W(dx) = \frac{1}{\phi(\theta) }
\label{useK}
\end{equation}
for $\Re(\theta)>0$.

\item We choose the process $H$ to be such that its semigroup $\mathbb{P}(H_t\in d x)$ is known 
in explicit form and then make use of the following identity
(see for example \citet{Ber1996} or \citet{Kyp2006}),
\begin{equation}
\int_0^\infty d t \cdot \mathbb{P}(H_t \in d x) = W( dx) , \, x\geq 0, 
\label{resolvent2}
\end{equation}
whenever $X$ does not drift to $-\infty$.
\end{enumerate}

\noindent Naturally, forcing a choice of descending Ladder height process, or equivalently  $\phi$, requires one to know that a 
spectrally negative \Levy{} process, $X$, exists whose Laplace exponent respects the
factorization (\ref{WHfact}). The next Theorem provides the necessary justification. Indeed, to some extent it shows  how to construct a L\'evy process with a given descending ladder height process as well as a prescribed ascending ladder height process. For spectrally negative L\'evy processes, the ascending ladder height process is a (killed) linear unit drift and has only a single parameter, namely $\Phi(0)$, the killing rate. Hence the construction in the below theorem offers a parameter $\varphi$ which plays the role of $\Phi(0)$. Since both ascending and descending ladder height process cannot both be killed one also sees the parameter restriction 
\begin{equation}
\varphi\kappa =0
\end{equation} 
where $\kappa$ is the killing rate of the descending ladder height process.

\begin{theorem}\label{inverse}
Suppose that $H$ is a subordinator, killed at rate $\kappa\geq 0$, with 
\Levy{} measure which is absolutely continuous with non-increasing density and drift $\zeta$. Suppose further that $\varphi\geq 0$ is 
given such that $\varphi\kappa=0$. Then
there exists a spectrally negative \Levy{} process  $X$, henceforth referred to as the `parent process', 
such that for all $x\geq 0$, $\mathbb{P}(\tau^+_x <\infty)=e^{-\varphi x}$ and whose descending ladder 
height process is precisely the process $H$. The \Levy{} triple $(a,\sigma, \Pi)$ of the parent process 
is uniquely identified as follows.
The Gaussian coefficient is given by $\sigma = \sqrt{2\zeta}$. The \Levy{} measure is given by 
\begin{equation}
\Pi(-\infty, -x) = \varphi\Upsilon(x,\infty)+ \frac{ d \Upsilon(x)}{d x}.
\label{density}
\end{equation}
Finally 
\begin{equation}
a =\int_{(-\infty, -1)} x\Pi(d x)  - \kappa
\end{equation}
if $\varphi=0$ and otherwise when $\varphi>0$ we can establish the value of $a$ from the equation
\begin{equation}
a\varphi = \frac{1}{2}\sigma^2\varphi^2 + \int_{(-\infty, 0)}(e^{\varphi x} -1- x\varphi\mathbf{1}_{\{x>-1\}})\Pi(d x).
\end{equation}
In all cases, the Laplace exponent of the parent process is also given by
\begin{equation}
\psi(\theta) = (\theta-\varphi)\phi(\theta)
\end{equation}
for $\theta \geq 0$ where $\phi(\theta) = - \log \mathbb{E}(e^{-\theta H_1})$. 
\end{theorem}
\begin{proof}
The proof is reasonably self evident given the preceding account of the ladder height process. 
One needs only the additional information that if $X$ is any spectrally negative \Levy{} process with 
\Levy{} measure $\Pi$ then 
\begin{equation}
\mathbb{E}(X_1) = -a + \int_{(-\infty, -1)} x\Pi(d x).
\end{equation}
Moreover when $\Phi(0)>0$ then necessarily the descending ladder height process has no killing and 
when the descending ladder height process is killed then $\Phi(0)=0$. Further, for all $x\geq 0$
$\mathbb{P}(\tau^+_x <\infty) =e^{-\Phi(0)x}$. 
\hfill\hfill\hfill$\square$\end{proof}

\bigskip

The idea of working `backwards' through the Wiener-Hopf factorization as we have done above can also be found in \citet{BRY} and \citet{Vigonthesis}.
Note that it is more practical to describe the parent process in terms of the triple $(\sigma,\Pi,\psi)$ than the triple $(a, \sigma, \Pi)$ and we shall frequently do this in the sequel.
It is also worth making an observation for later reference concerning the path variation of the 
process $X$ for a given a descending ladder height process $H$.

\begin{corollary}\label{variation?}
 Given a killed subordinator $H$ satisfying the conditions of the previous Theorem, 
\begin{description}
 \item[(i)]  the parent process has paths of unbounded variation if and only if $\Upsilon(0,\infty)=\infty$ or 
 $\zeta>0$,
\item[(ii)] if $\Upsilon(0,\infty)=\lambda <\infty$ then the parent process necessarily decomposes in the form 
\begin{equation}\label{driftterm}
X_t=(\kappa+\lambda-\zeta\varphi)t+\sqrt{2\zeta}B_t-S_t
\end{equation}
where $B=\{B_t: t\geq 0\}$ is a Brownian motion, $S=\{S_t : t\geq 0\}$ is an independent driftless 
subordinator with \Levy{} measure $\nu$ satisfying $\nu(x,\infty) = \Pi(-\infty, -x)$.
\end{description}
\end{corollary}

\begin{proof} The path variation of $X$ follows directly from (\ref{density}) and the fact that $\sigma=\sqrt{2\zeta}$. Also using (\ref{density}), the Laplace exponent of the decomposition (\ref{driftterm}) can be computed as follows with the help of an integration by parts;
\begin{eqnarray*}
\lefteqn{(\kappa + \lambda -\zeta\varphi)\theta + \zeta\theta^2 - \varphi\theta\int_0^\infty e^{-\theta x}\Upsilon(x,\infty)dx - \theta\int_0^\infty e^{-\theta x}\frac{d\Upsilon}{dx}(x)dx}&&\\
&=& (\kappa + \Upsilon(0,\infty) -\zeta\varphi)\theta + \zeta\theta^2 - \varphi\int_0^\infty (1-e^{-\theta x})\frac{d\Upsilon}{dx}(x)dx - \theta\int_0^\infty e^{-\theta x}\frac{d\Upsilon}{dx}(x)dx\\
&=&(\theta-\varphi)\left(\kappa + \theta \zeta + \int_0^\infty (1-e^{-\theta x})\frac{d\Upsilon}{dx}(x)dx\right).
\end{eqnarray*}
This agrees with the Laplace exponent $\psi(\theta)=(\theta-\varphi)\phi(\theta)$ of the parent process constructed in Theorem \ref{inverse}. 
\hfill\hfill\hfill$\square$\end{proof}

\section{The Gaussian tempered stable convolution class}\label{class1}
In this paper we introduce a new 
family of spectrally negative L\'evy processes from which our new examples of scale 
functions will emerge.
We call them  {\em Gaussian tempered stable convolution}, GTSC for short.
When there is no Gaussian part we call the distribution a tempered stable convolution
and write TSC for short.
\subsection{The tempered stable ladder process}
The starting point for the construction is a {\em tempered stable subordinator}
plus a linear drift, possibly killed, that will play the role of the 
descending ladder height process for the GTSC parent process.
Some references for tempered stable distributions and tempered stable \Levy{} processes
are \cite{steutelvanharn}, \cite{cont} and \cite{schoutens}. Note, that several names and origins, and many different
parameterisations are used for the tempered stable distributions and processes.

In our parametrization, the tempered stable subordinator involves three parameters,
the {\em stability parameter} $\alpha<1$, 
the {\em tempering parameter} $\gamma\geq0$, and the {\em scaling parameter} $c>0$.
When $\alpha\leq0$ it is required that $\gamma>0$. Furthermore we might add a linear
drift with rate $\zeta\geq0$ to the process, and possibly kill the process at 
rate $\kappa\geq0$.

The Laplace exponent of the tempered stable subordinator, and henceforth the descending ladder 
height process is thus taken to be$^1$\footnote{$^1$Note that we 
use $\Gamma(z)$ as a meromorphic function 
with simple poles at the non-positive integers.
By the functional equation of the gamma function we have
$\Gamma(-\alpha)=-\alpha^{-1}\Gamma(1-\alpha)$, and $\Gamma(-\alpha)<0$ 
for $\alpha\in(0,1)$.}
\begin{equation}\label{ts-phi}
\phi(\theta)=
\kappa+\zeta\theta+c\Gamma(-\alpha)(\gamma^\alpha-(\gamma+\theta)^\alpha),
\qquad\Re(\theta)>-\gamma,
\end{equation}
and the associated \Levy{} measure is given by
\begin{equation}\label{ts-xi}
\Upsilon(dx)=cx^{-\alpha-1}e^{-\gamma x}dx\qquad(x>0).
\end{equation}
For $0\leq\alpha<1$ the process has infinite activity.
For $\alpha=0$ the expression (\ref{ts-phi}) is to be understood in a limiting sense, i.e.,
$\phi(\theta)=\kappa+\zeta\theta-c\log\left(\gamma/(\gamma+\theta)\right)$, and
the process is simply a (killed) gamma subordinator (with drift).
If $\alpha<0$ the process is a compound Poisson process with intensity parameter
$c\Gamma(-\alpha)\gamma^\alpha$ and gamma distributed jumps with $-\alpha$ degrees of freedom 
and exponential parameter $\gamma>0$.
\subsection{The associated parent process
}\label{TSparent}
We may now invoke Theorem~\ref{inverse} to construct the associated GTSC process.
This introduces another parameter $\varphi\geq0$. 
To be able to apply the theorem, we need a decreasing \Levy{} density, and thus
we have to restrict the scaling parameter to $-1\leq\alpha<1$.
The theorem tells us that the Laplace exponent of the parent process is 
\begin{equation}\label{psiTSC}
\psi(\theta)=(\kappa-\varphi\zeta)\theta+\zeta\theta^2
+c(\theta-\varphi)\Gamma(-\alpha)(\gamma^\alpha-(\gamma+\theta)^{\alpha}).
\end{equation}
for $\Re(\theta)>\gamma$, and obviously for $\alpha=0$ we understand the above expression in the limiting sense so that $\psi(\theta)=(\kappa+\zeta\varphi)\theta+\zeta\theta^2
-c(\theta+\varphi)\log\left({\gamma}/{\gamma+\theta}\right)$. It is important to recall here and throughout the remainder of the paper that $\kappa\varphi=0$.
The corresponding \Levy{} measure given by 
\begin{equation}\label{cmconv}
\Pi(dx)=c\frac{(\varphi+\gamma)}{(-x)^{\alpha+1}}e^{\gamma x}dx+c\frac{(\alpha+1)}{(-x)^{\alpha+2}}e^{\gamma x}dx
\end{equation}
for $x<0$. This
indicates that the jump part is the result of the independent sum of 
two spectrally negative tempered stable processes with stability parameters $\alpha$ and $\alpha+1$.
We also note from Theorem \ref{inverse} that $\sigma=\sqrt{2\zeta}$, indicating the presence of a Gaussian component.  
This motivates the choice of terminology {\em Gaussian Tempered Stable Convolution}.

If $0<\alpha<1$ the jump component is the sum of an infinite activity negative tempered stable subordinator
and an independent spectrally negative tempered stable process with infinite variation. 
If $\alpha=0$ the jump component is the sum of 
a spectrally negative infinite variation tempered stable process with stability parameter $1$ and 
exponential parameter $\gamma$
and the negative of a gamma subordinator with exponential parameter $\gamma$.
If $-1\leq\alpha<0$ 
the jump part of the parent process is the independent sum 
tempered stable subordinator with stability parameter $1+\alpha$ 
and exponential parameter $\gamma$,
and an independent negative compound Poisson subordinator with
 jumps from a gamma distribution with $-\alpha$ degrees of freedom
and exponential parameter $\gamma$. 
In the extreme case $\alpha=-1$, the parent process has negative jumps which are compound Poisson and 
exponentially distributed with parameter $\gamma$.
\section{Evaluating GTSC scale and $q$-scale functions}
\subsection{The case with rational nonzero stability parameter}
\begin{theorem}\label{rational}
Suppose $\alpha=m/n$
with
$m\in\mathbb Z$, $n\in\mathbb Z$, $0<|m|<n$, and $\gcd(|m|,n)=1$.
Let us consider the GTSC process with Laplace exponent
\begin{equation}
\psi(\theta)=
(\theta-\varphi)
\left[
\kappa+\zeta\theta+c\Gamma(-\alpha)\left(\gamma^\alpha-(\gamma+\theta)^\alpha\right)\right].
\end{equation}
Let $m_+=\max(m,0)$ and $m_-=\max(-m,0)$. Then the polynomial
\begin{equation}
f_q(z)=
(z^n-\gamma-\varphi)
\left[
(\kappa+\zeta\gamma+c\Gamma(-\alpha)\gamma^\alpha)z^{m_-}+\zeta z^{n+m_-}-c\Gamma(-\alpha)z^{m_+})
\right]-qz^{m_-}
\end{equation}
has at least one real root.
Let $\ell$ denote the number of distinct roots,
$r_1,\ldots,r_\ell$ the roots, arranged such that $r_1$ is the largest real root, 
and $\mu_1,\ldots,\mu_\ell$ their multiplicities.
Let $A_{kj}$ for $k=1,\ldots,\ell$, $j=0,\ldots,\mu_k-1$ denote the 
coefficients in the partial fraction decomposition
\begin{equation}
\frac{z^{m_-}}{f_q(z)}=
\sum_{k=1}^\ell\sum_{j=0}^{\mu_k-1}\frac{A_{kj}}{(z-r_k)^{j+1}}.
\end{equation}
Then we have
\begin{equation}\label{Phiq}
\Phi(q)=r_1^n-\gamma
\end{equation}
and
\begin{equation}\label{main-formula}
W^{(q)}(x)=e^{-\gamma x}\sum_{k=1}^\ell\sum_{j=0}^{\mu_k-1}A_{kj}
\frac1{j!}x^{(j+1)/n-1}
E_{\frac1n,\frac1n}^{(j)}(r_kx^{\frac1n}).
\end{equation}
\end{theorem}
\begin{proof} We have 
\begin{equation}\label{f-psi}
f_q(z)=z^{m_-}\left(\psi(z^n-\gamma)-q\right)
\end{equation}
for $z\in\mathbb Z$ with $|\arg(z)|<\pi/n$, and conversely, 
\begin{equation}\label{psi-f}
\frac1{\psi(\theta)-q}=\frac{(\theta+\gamma)^{\frac{m_-}{n}}}{f_q((\theta+\gamma)^{\frac1n})}
\end{equation}
for $\theta\in\mathbb C\setminus(-\infty,-\gamma])$ such that the denominator on the
left hand side does not vanish.
We know from the theory of scale functions that $\theta=\Phi(q)$ is a real root
of $\psi(\theta)-q=0$, and there is no larger real root of that equation.
Thus $z=(\Phi(q)+\gamma)^{\frac1n}$ is a real root of $f_q(z)=0$ and there is no
larger real root of that equation. This shows~(\ref{Phiq}).

The left hand side of (\ref{psi-f}) is the Laplace transform of $W^{(q)}(x)$, and we know it
is an analytic function for $\Re(\theta)>\Phi(q)$. 
Using the partial fraction decomposition we get 
\begin{equation}\label{tscLT}
\int_0^\infty e^{-\theta x}W^{(q)}(x)dx=
\sum_{k=1}^\ell \sum_{j=0}^{\mu_k-1}\frac{A_{kj}}{((\gamma+\theta)^{\frac1n}-r_k)^{j+1}},
\end{equation}
and this equation is valid for any $q\geq 0$ and $\Re(\theta)>\Phi(q)$.
We recognize the Laplace transform of exponentially tilted derivatives of Mittag-Leffler
functions on the right hand side, see for example \cite[Prop.7.1.9, p.359]{JacIII},
and obtain (\ref{main-formula}).
\hfill\hfill$\square$\end{proof}

\bigskip

Generally speaking this is rather an awkward formula to work with analytically. 
The main strength of the expression lies with it being a simple matter to program 
into a package such as MATLAB or Mathematica. 
It is instructive to revisit the (known) analytical properties 
of $W(x)$ and $W^{(q)}(x)$ listed in section \ref{known-analytical} and check them 
directly for our explicit expression as this will yield in most cases more detail.

\bigskip

\noindent{\bf  Smoothness:}  Since the Mittag-Leffler functions are entire functions,
we see that $W^{(q)}(x)$, which is {\it a priori} defined for $x\in(0,\infty)$,
admits an analytic continuation
to $x\in\mathbb{C}\setminus(-\infty,0]$, and thus $W^{(q)}$ is $C^\infty$ on $(0,\infty)$.

\bigskip

\noindent{\bf  Behaviour at zero:} We obtain the behaviour at zero from the power series expansion of the derivatives 
of the Mittag-Leffler function. Exploiting algebraic relations of the $r_k$ and $A_{kj}$
several terms cancel and we obtain for $\zeta>0$
\begin{equation}
W^{(q)}(x)\sim\frac{x}{\zeta},\qquad
W^{(q)\prime}(x)\sim\frac1\zeta
\qquad(x\to0).
\end{equation}
A similar argument applies for $\zeta=0$. If $0<\alpha<1$ we get 
\begin{equation}
W^{(q)}(x)\sim -\frac{x^\alpha}{c\Gamma(-\alpha)\Gamma(1+\alpha)},\qquad 
W^{(q)\prime}(x)\sim-\frac{ x^{\alpha-1}}{c\Gamma(-\alpha)\Gamma(\alpha)},\qquad (x\to0).
\end{equation}
In the case that $-1<\alpha<0$ and $\zeta=0$.
we get again by cancellation of terms and limiting behaviour at zero of Mittag-Leffler functions
\begin{equation}
W^{(q)}(x)\sim\frac{1}{\kappa+c\gamma^\alpha\Gamma(-\alpha)},\qquad 
W^{(q)\prime}(x)\sim\frac{cx^{-\alpha-1}}{(\kappa+c\gamma^\alpha\Gamma(-\alpha))^2},\qquad (x\to0).
\end{equation}
Note, another way to establish these last two results is by Karamata's Tauberian Theorem 
and the Monotone Density Theorem respectively from the asymptotics 
of $\psi(\theta)$ as $\theta\to\infty$.
This approach shows, that the asymptotics above also hold for irrational $\alpha\in(-1,1)$.

\bigskip

\noindent{\bf  Behaviour at infinity:} 
The behaviour at infinity can be obtained from the asymptotics of the Mittag-Leffler functions
at infinity. The dominating contribution comes from the term with $r_1$.
It is useful to recall that $f_q(z)=\psi(z^n-\gamma)-q$ and 
thus $f_q'(z)=nz^{n-1}\psi'(z^n-\gamma)$.
Let us consider first $q=0$.
Note, that we have $\psi'(0+)=\kappa-\varphi(\zeta+c\gamma^{\alpha-1}\Gamma(1-\alpha))$.
Suppose $\psi'(0+)>0$, which implies that $\varphi=0$. Then $r_1=\gamma^{\frac1n}$
is a single root of $f_0(z)$. 
After some elementary simplifications we obtain
\begin{equation}
W(x)\sim\frac{1
}{\kappa}\qquad(x\to\infty)
\end{equation}
in agreement with the theory.
Suppose now $\psi'(0+)<0$. Then $r_1=(\varphi+\gamma)^{\frac1n}$ is a simple root of $f_0(z)$.
As above, using the asymptotics of the Mittag-Leffler functions at infinity,
we see that the term with $r_1$ dominates all other terms.
After elementary simplifications we obtain 
\begin{equation}
W(x)\sim \frac{
e^{\varphi x}}{\kappa+\zeta\varphi+c\Gamma(-\alpha)(\gamma^\alpha - (\gamma+\varphi)^\alpha)}
\qquad(x\to\infty).
\end{equation}
Consider  the  third case, $\psi'(0+)=0$. This can only happen when $\kappa=0$ 
and $\varphi=0$. Then $r_1=\gamma^{\frac1n}$ 
is a double root of $f_0(z)$.
The asymptotics show, that the contribution from $r_1$ dominates
and we obtain
\begin{equation}
W(x)\sim\frac{
x}{\zeta+c\gamma^{\alpha-1}\Gamma(1-\alpha)}
\qquad(x\to\infty).
\label{becauseofrenewal}
\end{equation}

\noindent Finally in the case $q>0$, little more can be said than (\ref{Wq(infty)}).
\subsection{The case with general nonzero stability parameter $\alpha\in(-1,1)$}\label{CSBexample}
When we take $q=0$ and $\zeta=0$ and $\alpha\in(-1,1)\backslash\{0\}$ 
(without the restriction of begin a rational number) one may obtain much cleaner expressions for the scale function 
than the formulation in Theorem \ref{rational}.
\begin{theorem}\label{niceform}
Suppose $\alpha\in(-1,1)\setminus\{0\}$ and consider a TSC process (without
Gaussian component) and Laplace exponent
\begin{equation}
\psi(\theta)=(\theta-\varphi)\left[\kappa+c\Gamma(-\alpha)
\left(\gamma^\alpha-(\gamma+\theta)^\alpha\right)\right].
\end{equation}
Then 
If $0<\alpha<1$ 
then
\begin{equation}
W(x)=-\frac{e^{\varphi x}}{c\Gamma(-\alpha)}\int_0^xe^{-(\gamma+\varphi)y}y^{\alpha-1}
E_{\alpha,\alpha}\left(\frac{\kappa+c\Gamma(-\alpha)\gamma^\alpha y^\alpha}{c\Gamma(-\alpha)}\right)dy,
\end{equation}
if $-1<\alpha<0$, then
\begin{eqnarray}
\hspace{-1cm}W(x)&=&
\frac{e^{\varphi x}}{\kappa+c\Gamma(-\alpha)\gamma^\alpha}\notag\\
&&+\frac{c\Gamma(-\alpha)e^{\varphi x}}{(\kappa+c\Gamma(-\alpha)\gamma^\alpha)^2}\int_0^xe^{-(\gamma+\varphi)y}y^{-\alpha-1}
E_{-\alpha,-\alpha}\left(\frac{c\Gamma(-\alpha)y^{-\alpha}}{\kappa +c\Gamma(-\alpha)\gamma^\alpha}\right)dy.
\end{eqnarray}
\end{theorem}
\begin{proof} This follows from the known Laplace transforms
\begin{equation}
\int_0^\infty e^{-\theta x}x^{\alpha-1}E_{\alpha,\alpha}(\lambda x^\alpha)dx=\frac1{\theta^\alpha-\lambda}
\end{equation}
and
\begin{equation}
\int_0^\infty e^{-\theta x}\lambda^{-1}x^{-\alpha-1}E_{-\alpha,-\alpha}(\lambda^{-1} x^{-\alpha})dx=
\frac\lambda{\lambda-\theta^\alpha}-1,
\end{equation}
valid for $\alpha>0$ resp.\ $\alpha<0$ together with the 
the well-known rules for Laplace transforms concerning, primitives and tilting.
\hfill\hfill$\square$\end{proof}

\bigskip

With these  closed-form expressions, the explicit asymptotics discussed after Theorem \ref{rational} are immediate by inspection. Note in the particular case that $\kappa=-c\Gamma(-\alpha)\gamma^\alpha $ the expression for the scale function in the case $0<\alpha<1$ reduces simply to 
\begin{equation}
W(x)=\frac{1}{\kappa}\int_0^x 
\frac{\gamma^\alpha}{\Gamma(\alpha)} y^{\alpha-1} e^{-\gamma y}dy.
\end{equation}

Note that although the Laplace inversions involved in the proof of Theorem \ref{niceform} are straightforward, it is not clear that the resulting expressions for $W$ are scale functions without the presence of Theorem \ref{inverse}.

\subsection{The case with $\alpha=\frac12$:  inverse Gaussian descending ladder height process\label{ig-special}}
We will assume now that $\zeta=0$, $\varphi=0$, $\kappa=0$.
Substituting $\alpha\mapsto1/2$, $c\mapsto\delta/\sqrt{2\pi}$, 
$\gamma\mapsto\gamma^2/2$ reveals that the tempered stable subordinator we have been working with 
is in fact the familiar inverse Gaussian $IG(\delta,\gamma)$-subordinator
when $\alpha=1/2$. It turns out for this case that we also get some cleaner expressions than the one given in Theorem \ref{rational}.


When the descending ladder height process 
is the $IG(\delta,\gamma)$ subordinator,
the jump part of the parent process 
is a superposition of a negative IG subordinator and a spectrally 
negative tempered stable process with stability parameter~$3/2$.

Following the general approach in Theorem \ref{rational} 
we relate $\psi(\theta)$ to a polynomial, which, in the
IG parametrization of this section, is $f_q(z)=\psi(z^2-\gamma^2/2)-q$.
The first simplification is, that $f_q(z)$ is a polynomial of degree~4
if $\zeta>0$, and degree~3, if $\zeta=0$.
Thus completely explicit, elementary (and somewhat lengthy) 
expressions of its roots
can be given in terms of radicals by the formul\ae{} of Ferrari resp.\ Cardano.
The second simplification concerns the Mittag-Leffler functions.
As $\alpha=1/2$ the Mittag-Leffler function simplifies to 
an expression involving the more familiar (complementary) error function.

We find that $f_0(z)$ has a positive double a negative single root,
and after some elementary simplifications,
\begin{equation}\label{igw0}
W(x)=
\frac1{2\delta\gamma}\left[
(1+\gamma^2x)\erfc\left(-\gamma\sqrt{x/2}\right) 
+\gamma\sqrt{2x/\pi}e^{-\frac12\gamma^2x}-1
\right].
\end{equation}
Actually, in this special case the result can be verified faster by calculating the
Laplace transform of~(\ref{igw0}).
As a side remark, we note, 
that $W''(x)=-x^{-3/2}e^{-\frac12\gamma^2x}/(2\sqrt{2\pi}\delta)$
and thus we see, that $W$ is indeed concave, as we already know from the theory.
Let us now consider $q>0$ and put
\begin{equation}
q_0=\textfrac{16}{27}\delta\gamma^3.
\end{equation}
If $0<q<q_0$ then we have three simple real roots $r_1,r_2,r_3$, that can be obtained by Cardano's 
formula$^2$\footnote{$^2$Explicitly
$$\begin{array}{l}
r_1=\frac{\gamma}{3\sqrt{2}}+\frac{2\sqrt{2}\delta\gamma^2}{3\sqrt[3]{-8\delta^3\gamma^3+27\delta^2q+3\sqrt{3}\sqrt{27\delta^4q^2-16\delta^5\gamma^3q}}}+\frac{\sqrt[3]{-8\delta^3\gamma^3+27\delta^2q+3\sqrt{3}\sqrt{27\delta^4q^2-16\delta^5\gamma^3q}}}{3\sqrt{2}\delta}\\
r_2=\frac{\gamma}{3\sqrt{2}}-\frac{\sqrt{2}\left(1+i\sqrt{3}\right)\delta\gamma^2}{3\sqrt[3]{-8\delta^3\gamma^3+27\delta^2q+3\sqrt{3}\sqrt{27\delta^4q^2-16\delta^5\gamma^3q}}}-\frac{\left(1-i\sqrt{3}\right)\sqrt[3]{-8\delta^3\gamma^3+27\delta^2q+3\sqrt{3}\sqrt{27\delta^4q^2-16\delta^5\gamma^3q}}}{6\sqrt{2}\delta}\\
r_3=\frac{\gamma}{3\sqrt{2}}-\frac{\sqrt{2}\left(1-i\sqrt{3}\right)\delta\gamma^2}{3\sqrt[3]{-8\delta^3\gamma^3+27\delta^2q+3\sqrt{3}\sqrt{27\delta^4q^2-16\delta^5\gamma^3q}}}-\frac{\left(1+i\sqrt{3}\right)\sqrt[3]{-8\delta^3\gamma^3+27\delta^2q+3\sqrt{3}\sqrt{27\delta^4q^2-16\delta^5\gamma^3q}}}{6\sqrt{2}\delta}.
\end{array}
$$
If $q=\delta\gamma^3/2$ those formulae simplify, as one root is then zero.},
and
\begin{equation}
W^{(q)}(x)=e^{-\frac12\gamma^2x}\left[
\frac{r_1e^{r_1^2x}}{f_q'(r_1)}\erfc(-r_1\sqrt{x})
+\frac{r_2e^{r_2^2x}}{f_q'(r_2)}\erfc(-r_2\sqrt{x})
+\frac{r_3e^{r_3^2x}}{f_q'(r_3)}\erfc(-r_3\sqrt{x})
\right].
\end{equation}
If $q>q_0$ we have a simple real root $r_1$ and two complex conjugate simple roots $r_2$ and $r_3$.
The same formulas hold as in the previous case.
Alternatively, we could in the complex case write
\begin{equation}
W^{(q)}(x)=e^{-\frac12\gamma^2x}\left[
\frac{r_1e^{r_1^2x}}{f_q'(r_1)}\erfc(-r_1\sqrt{x})
+2\Re\left\{
\frac{r_2e^{r_2^2x}}{f_q'(r_2)}\erfc(-r_2\sqrt{x})
\right\}
\right].
\end{equation}
If $q=q_0$ then we have a simple positive and a double negative root, and after some easy
simplifications,
\begin{eqnarray}
W^{(q_0)}(x)&=&
\frac1{36\delta\gamma}\left[
6\gamma\sqrt{\frac{2x}{\pi}}e^{-\frac12\gamma^2x} 
+ 
15e^{\frac89\gamma^2x}\erfc\left(-\frac{5\gamma}{3}\sqrt{\frac{x}{2}}\right)\right. \notag\\
&&\hspace{3cm}
\left.- 
e^{-\frac49\gamma^2x}(15+2\gamma^2x)\erfc\left(\frac\gamma3\sqrt{\frac{x}{2}}\right)
\right].
\end{eqnarray}
\subsection{The stable case}
Because of their definition via their Laplace transform, it is immediate that scale functions are continuous in the parameters of the parent process. Taking $\gamma\downarrow 0$ in the Laplace exponent of the tempered stable ladder height process shows us that the parent process is an $\alpha+1$ stable process when $\kappa=\varphi=\zeta=0$ and $c=-1/\Gamma(-\alpha)$. Taking limits as $\gamma\downarrow0$ in the expression for the scale function in the first part of Theorem \ref{niceform} gives us the known result $W(x)  =
x^{\alpha}/\Gamma(\alpha+1)$.

\subsection{The case $\alpha=0$: gamma descending ladder height process}\label{case-gamma}
When $\alpha=0$ we can only treat the special case with $q=0$, 
$\kappa=0$, $\zeta=0$ and $\varphi=0$.
\begin{theorem}
\begin{equation}
W(x)=\int_0^\infty P(ct,\gamma x)dt,
\end{equation}
where $P(a,x)$ is the regularized lower incomplete gamma function
\end{theorem}
\begin{proof}
We may appeal to the second method in (\ref{resolvent2}) using resolvents   to note that 
the scale function density satisfies
\begin{equation}\label{Wprime0}
W'(x)=\frac1cx^{-1}e^{-\gamma x}\varphi(-\log(\gamma x))
\end{equation}
where
\begin{equation}
\varphi(\theta)=\int_0^{\infty}\frac{e^{-\theta x}}{\Gamma(x)}dx,
\end{equation}
which is the Laplace transform of the reciprocal gamma function.$^3$\footnote{$^3$
The value at $\theta=0$ is called the Frans\'en-Robertson constant and denoted
by $\varphi(0)=F$. A lot of analytical and numerical material on
$\varphi(\theta)$ can be found in \cite{FW1984}. That paper
contains also several references to related work going back 
to Paley, Wiener, Hardy and Ramanujan.
Note, that the reciprocal gamma distribution in \cite{FW1984} is not to be
confused with the inverse gamma distribution, which is often
called reciprocal gamma distribution as well.}
As $X$ has infinite variation, we have $W(0+)=0$.
Integrating (\ref{Wprime0}) yields the result.
\hfill\hfill$\square$\end{proof}


\subsection{Remarks on the cases $\alpha=-1/2$ and $\alpha=-1$}
The case $\alpha=-1/2$
is again covered in Theorem \ref{rational} above, but several simplifications occur.
The jump part of the ladder process is in this case a compound Poisson process with jumps from
the $\chi_1$-distribution. 
Again, we can avoid Mittag-Leffler functions and use the complimentary error function instead.
With regard to the roots of the polynomial equation we may say the following.
If $\zeta>0$ the polynomial $f_q(z)$ has degree~5 and its roots typically 
cannot be expressed in terms of radicals.
If $\zeta=0$ the polynomial is of degree~3 and Cardano's formula can be used.
If in addition $\kappa=0$ and $\varphi=0$ then we have
$r_1=\sqrt\gamma$ a double and $r_2=-\sqrt\gamma$ a single root.
We can proceed as is Section~\ref{ig-special} to obtain some simple closed
form expressions for $W(x)$ and $W^{(q)}(x)$.

When $\alpha  =-1$  the jump part of the parent process is
a compound Poisson process with negative exponential jumps,
and the ladder process is a compound Poisson process with exponential jumps.
The Laplace transform of the scale functions are rational,
and the scale functions can be obtained
directly from the partial fraction decomposition.
The results are implicitly contained in \cite{Mor} and \cite{KouWang} and we exclude the calculations here.

\section{Numerical illustrations}
Figure~\ref{fig1} contains graphs of the GTSC scale function $W(x)$ for stability parameters 
$\alpha=1/4$, $1/3$, $1/2$, $2/3$, $3/4$ in six cases.
Figure~\ref{fig2} contains graphs of the GTSC $q$-scale functions $W^{(q)}(x)$ with $q=1$, but otherwise
for the same parameters.
Below we introduce six cases considered in terms of the classification of the parent process.

\begin{itemize}
\item Case A, $\kappa=0$, $\varphi=0$, $\zeta=0$, $c=1$, $\gamma=1$:
The parent process is oscillating, has no diffusion part
and infinite variation jumps.
The ladder process is an infinite activity pure jump subordinator, has no linear drift, and is not killed.
\item Case B, $\kappa=1$, $\varphi=0$, $\zeta=0$, $c=1$, $\gamma=1$:
The parent process drifts to $+\infty$, has no diffusion part
and infinite variation jumps.
The ladder process is an infinite activity pure jump subordinator killed at rate $\kappa=1$.
\item Case C, $\kappa=0$, $\varphi=1$, $\zeta=0$, $c=1$, $\gamma=1$:
The parent process drifts to $-\infty$, has no diffusion part
and infinite variation jumps.
The ladder process is an infinite activity pure jump subordinator.
\item Case D, $\kappa=0$, $\varphi=0$, $\zeta=1$, $c=1$, $\gamma=1$:
The parent process is oscillating, has no linear drift, but a diffusion part
and infinite variation jumps.
The ladder process is an infinite activity pure jump subordinator plus a linear drift.
\item Case E, $\kappa=1$, $\varphi=0$, $\zeta=1$, $c=1$, $\gamma=1$:
The parent process drifts to $+\infty$, has a Gaussian component
and infinite variation jumps.
The ladder process is an infinite activity subordinator plus linear drift, killed at unit rate.
\item Case F, $\kappa=0$, $\varphi=1$, $\zeta=1$, $c=1$, $\gamma=1$:
The parent process drifts to $-\infty$, has a diffusion part
and infinite variation jumps.
The ladder process is an infinite activity pure jump subordinator plus a linear drift.
\end{itemize}
Now let us discuss the graphs in view of the theoretical properties listed in Section~\ref{known-analytical}.

\bigskip

\noindent{\bf  Smoothness:} All graphs look smooth, as the theory predicts.

\bigskip

\noindent{\bf  Concavity and convexity:} 
For $q=0$ we know in cases A, B,  D and E the graph is predicted to be concave. In cases $C$ and $F$ the graph is predicted to be convex-concave. Indeed this is the appearance in  
Figure~\ref{fig1}. 

For $q=1$ we know, that all graphs are concave-convex. We observe this behaviour in Figure~\ref{fig2}. In some cases, it is necessary to inspect the shape of the graph more closely in the neighbourhood of the origin in order to see concavity.

\bigskip

\noindent{\bf  Behaviour at zero:} 
For all cases we observe $W(0)=0$ in Figure~\ref{fig1} and $W^{(q)}(0)=0$
in Figure~\ref{fig2}, in agreement with formula~(\ref{W(0)}).

For the cases A--C (no diffusion part) we observe $W'(0)=+\infty$,
for the cases D--E (nonzero diffusion part) we observe $W'(0)=1$ 
in Figure~\ref{fig1}. We observe the same behaviour, 
$W^{(q)\prime}(0)=\infty$ resp.\ $W^{(q)\prime}(0)=1$, in Figure~\ref{fig2}.
This is in agreement with formula~(\ref{W'(0)}).

\bigskip

\noindent{\bf  Behaviour at infinity:}
Firstly let us consider the case $q=0$.
In cases A and D we have $\psi'(0+)=0$ and thus, according to (\ref{becauseofrenewal}), asymptotic linear growth as $x\to\infty$. 
In cases B and E we have $\kappa>0$ and hence $\psi'(0+)>0$ so that  $W(x)\to1/\kappa$ as $x\to\infty$.
In cases C and F we have $\varphi>0$ and hence $\psi'(0)<0$ and so there is exponential growth
of $W(x)$ as $x\to\infty$ according to the second case in formula (\ref{W(infty)}).
This corresponds to what is observed in the graphs in Figure~\ref{fig1}. (The behaviour for $x\to\infty$
becomes more prominent when plotting $0<x<20$.)

For the case $q>0$, all graphs in Figure~\ref{fig2} appear exhibit exponential growth as $x\to\infty$, 
in agreement with formula~(\ref{Wq(infty)}). 

\begin{figure}
\begin{tabular}{cc}
\includegraphics[width=7cm,height=6cm]{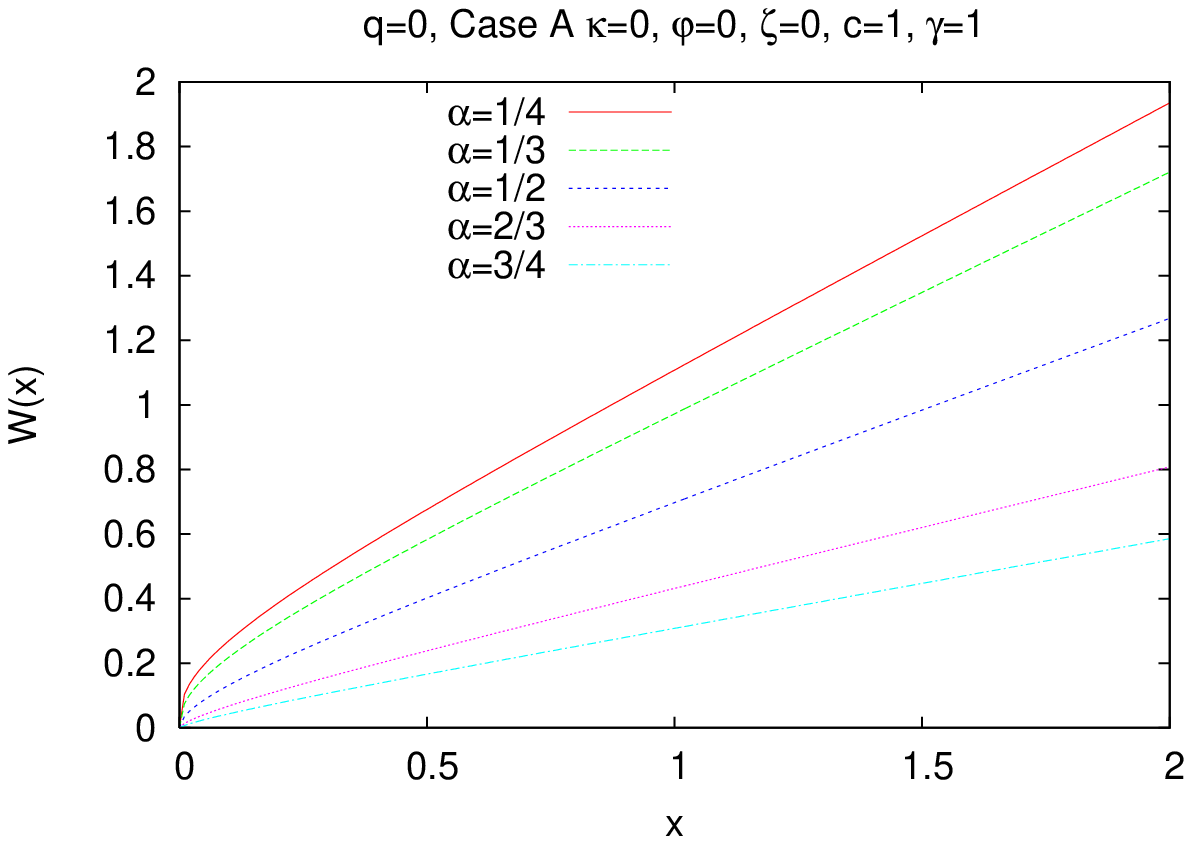}&\includegraphics[width=7cm,height=6cm]{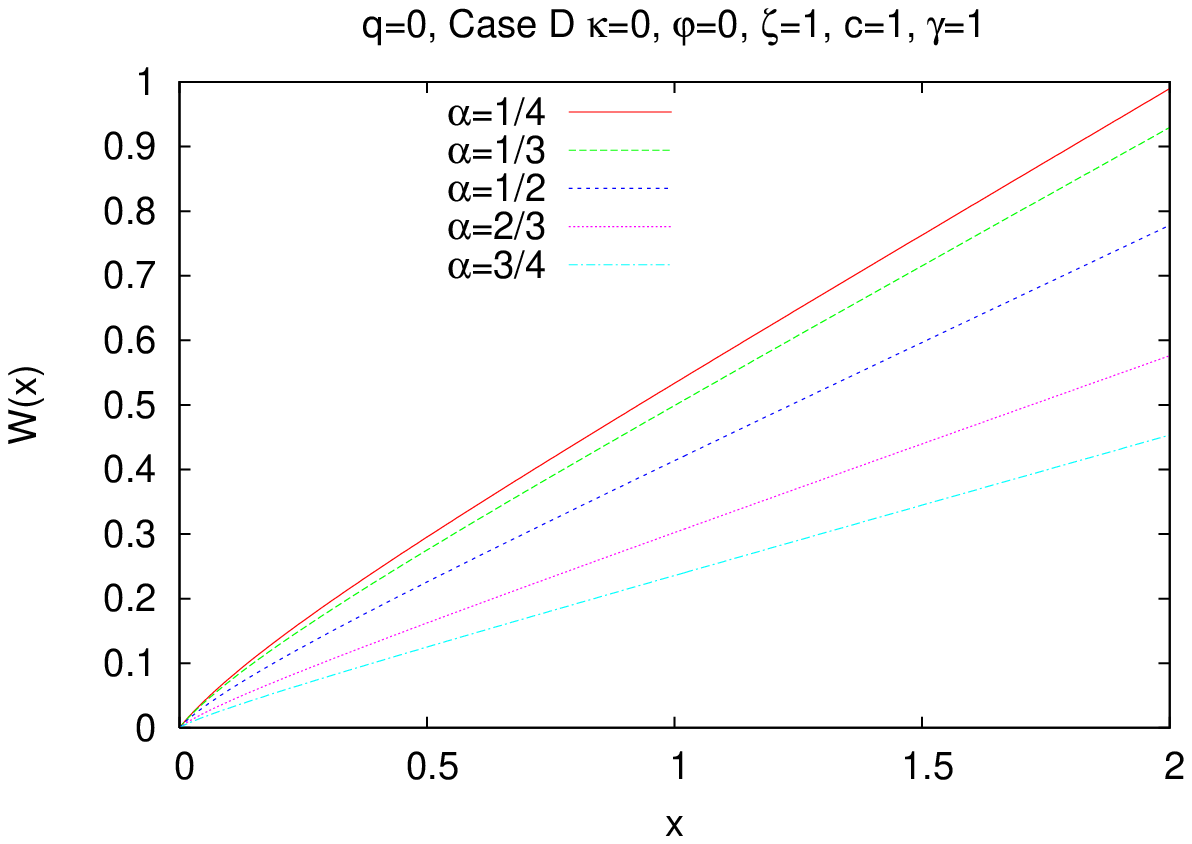}\\
\includegraphics[width=7cm,height=6cm]{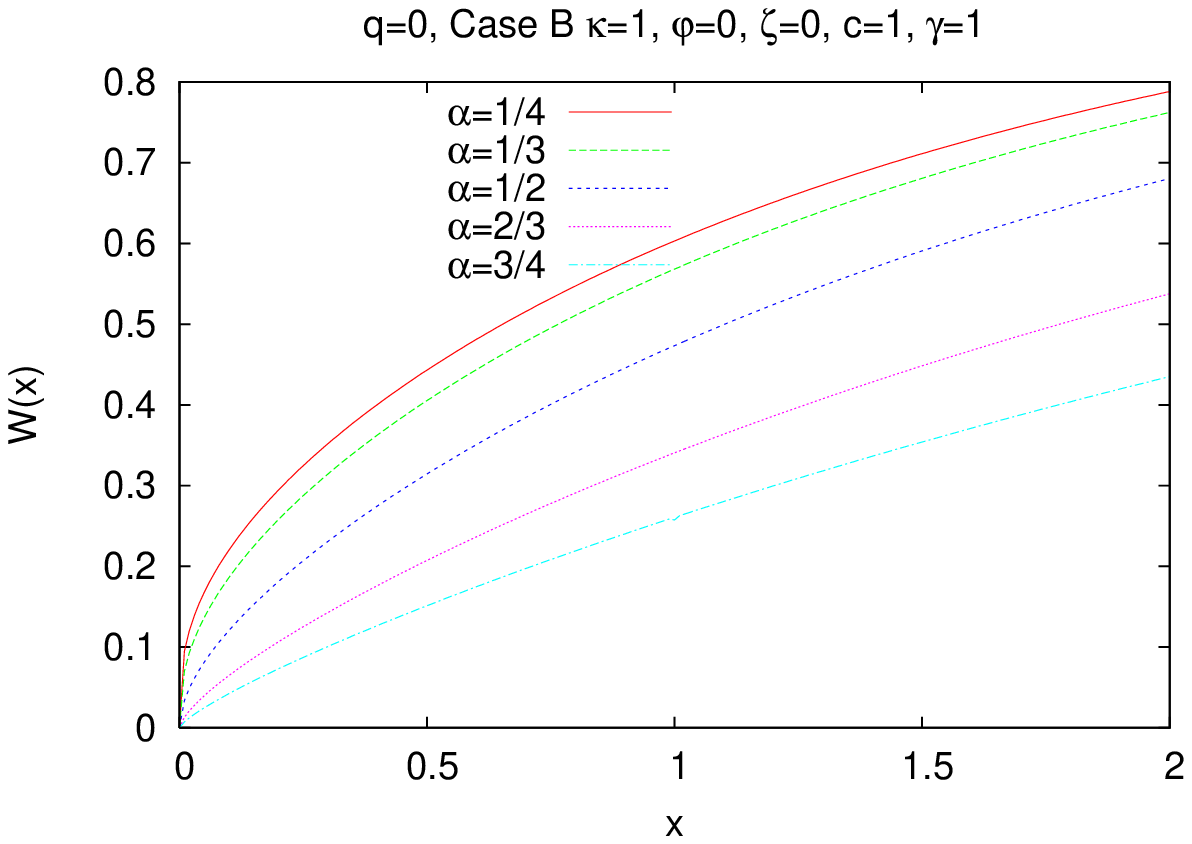}&\includegraphics[width=7cm,height=6cm]{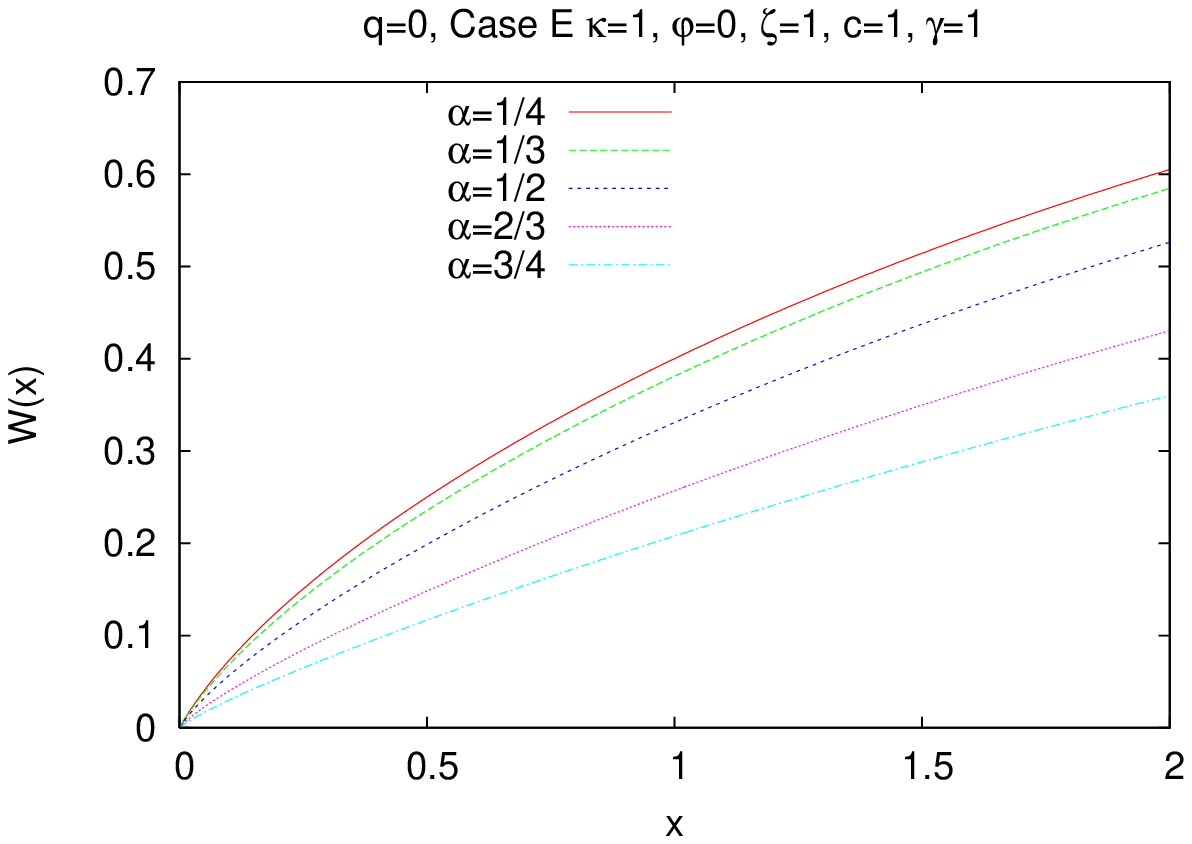}\\
\includegraphics[width=7cm,height=6cm]{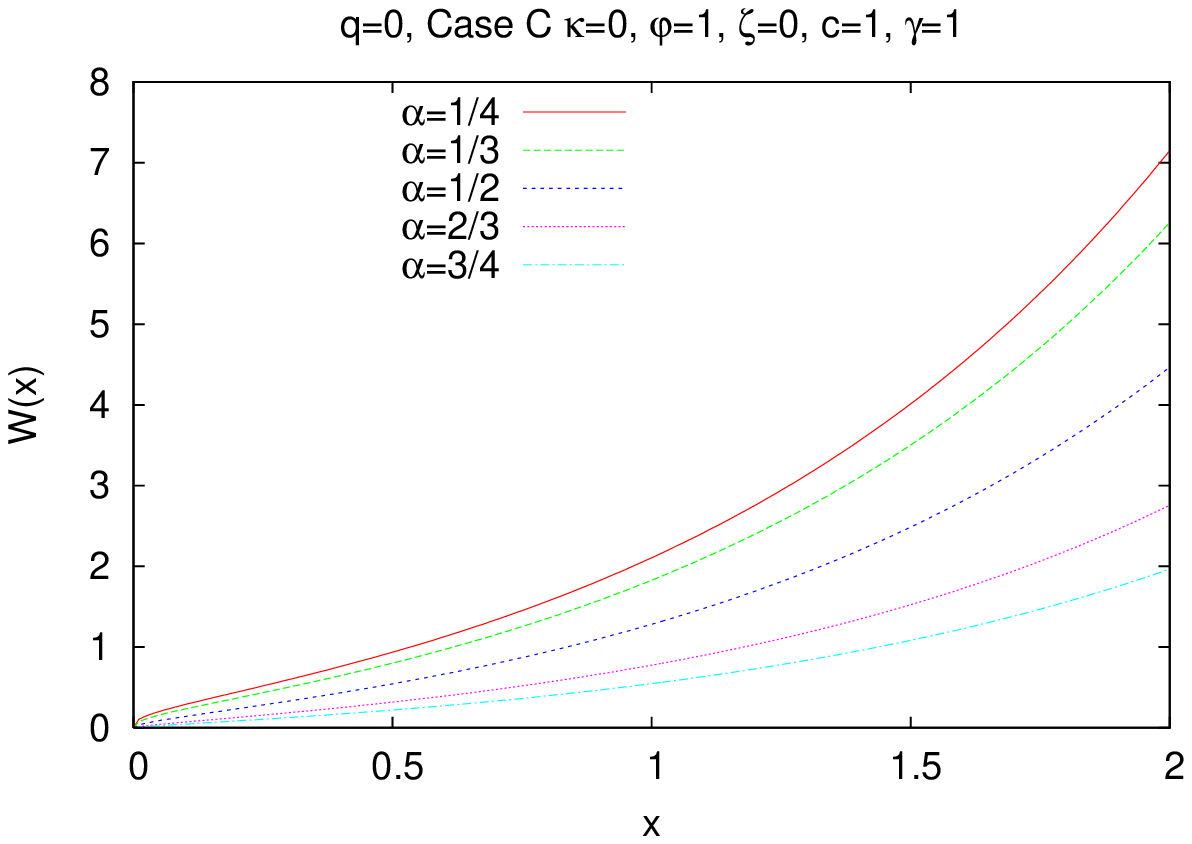}&\includegraphics[width=7cm,height=6cm]{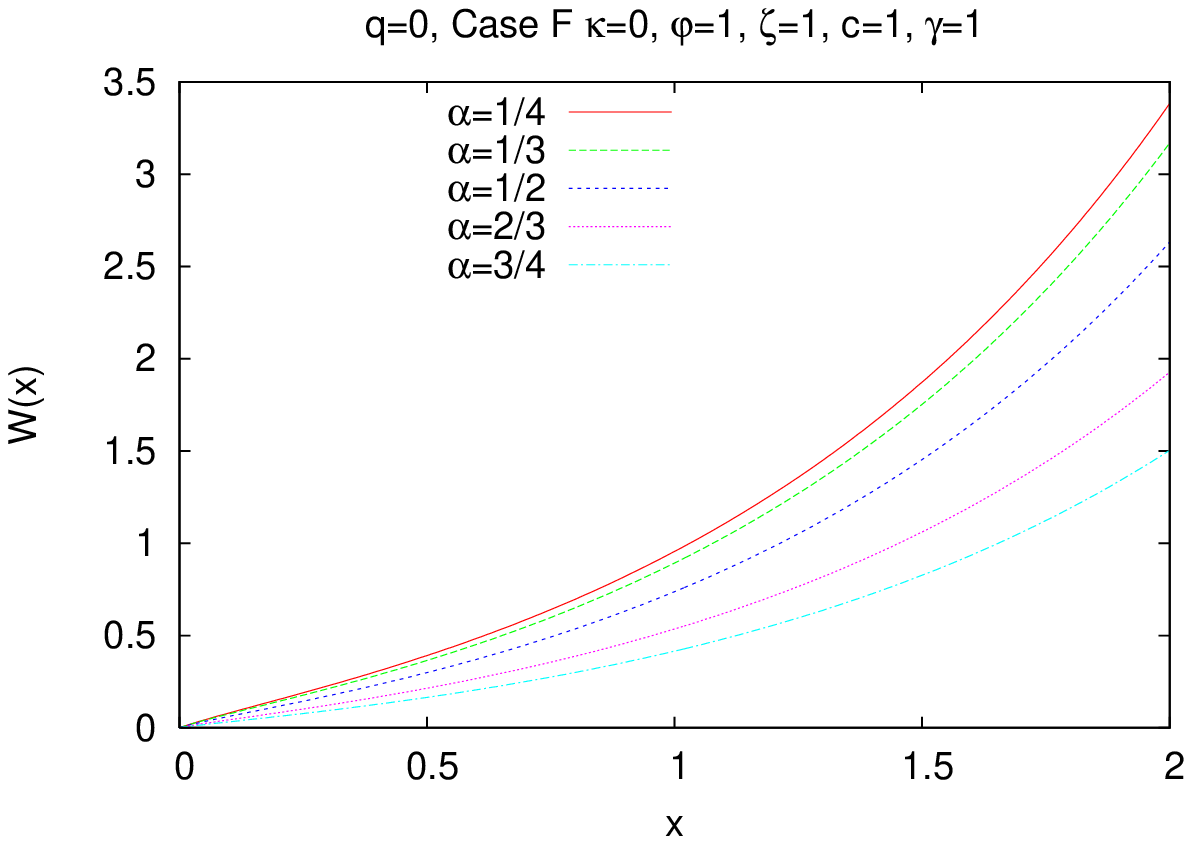}
\end{tabular}
\caption{Scale functions $W(x)$ for the GTSC class with $0<\alpha<1$}
\label{fig1}
\end{figure}
\begin{figure}
\begin{tabular}{cc}
\includegraphics[width=7cm,height=6cm]{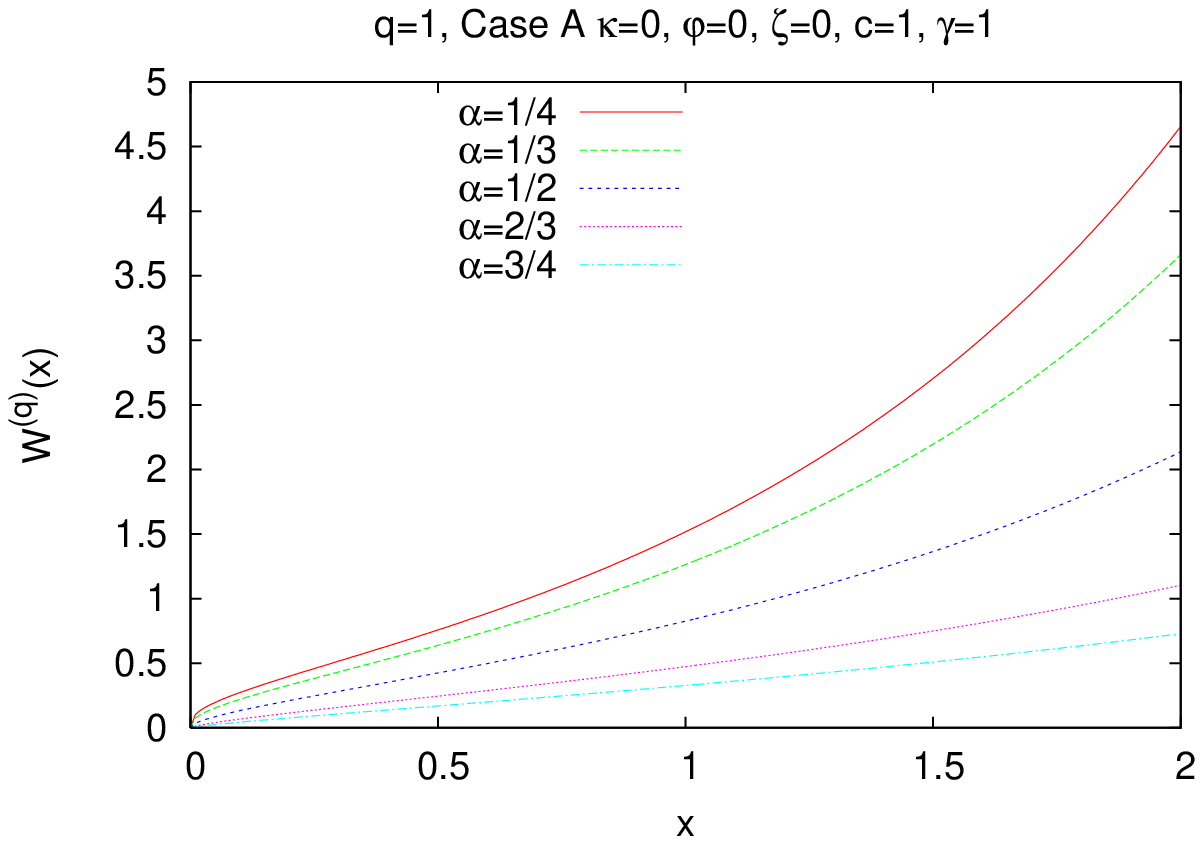}&\includegraphics[width=7cm,height=6cm]{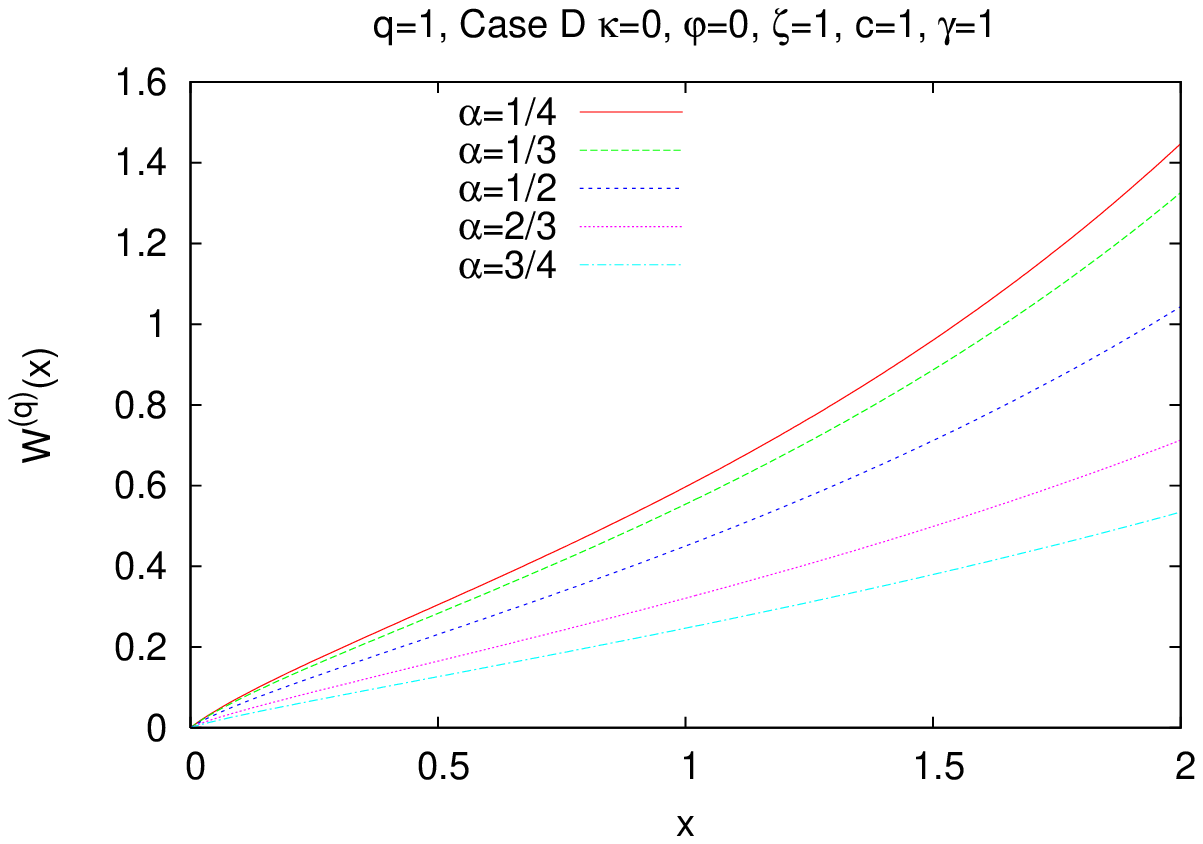}\\
\includegraphics[width=7cm,height=6cm]{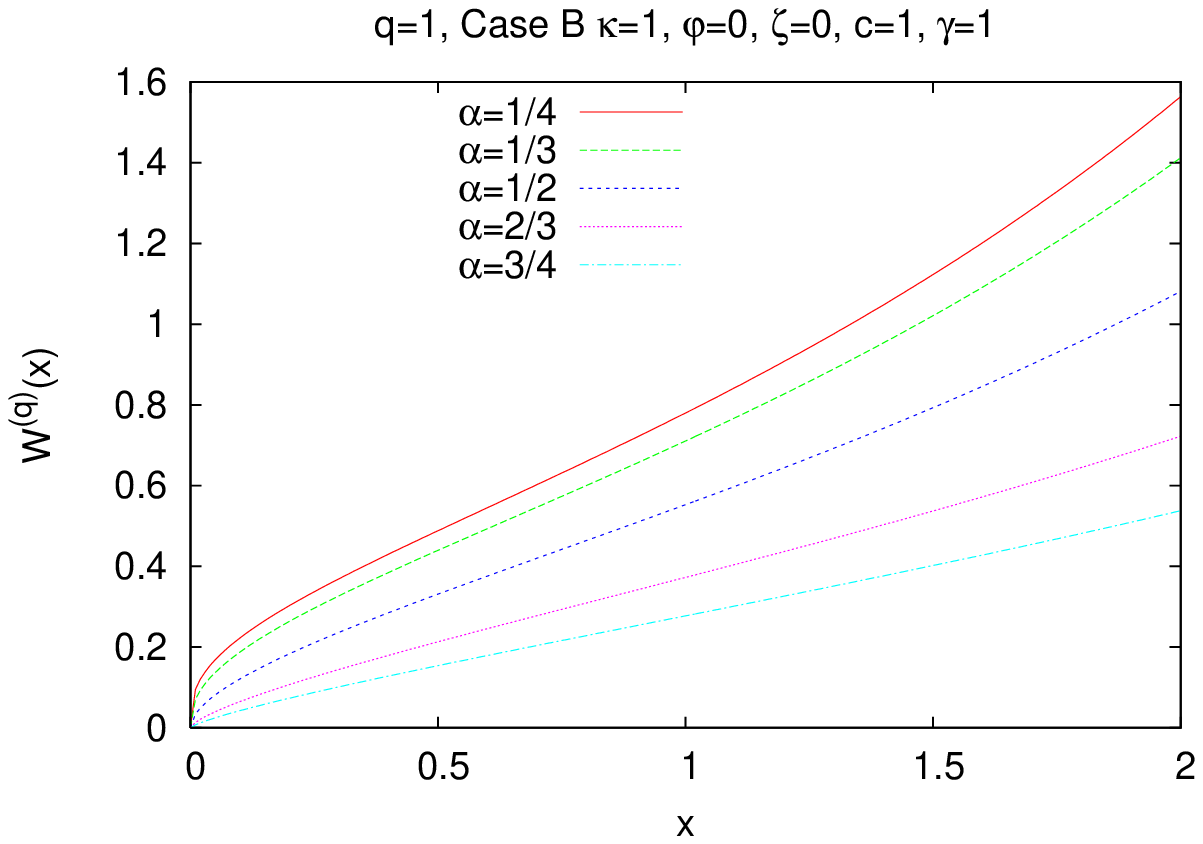}&\includegraphics[width=7cm,height=6cm]{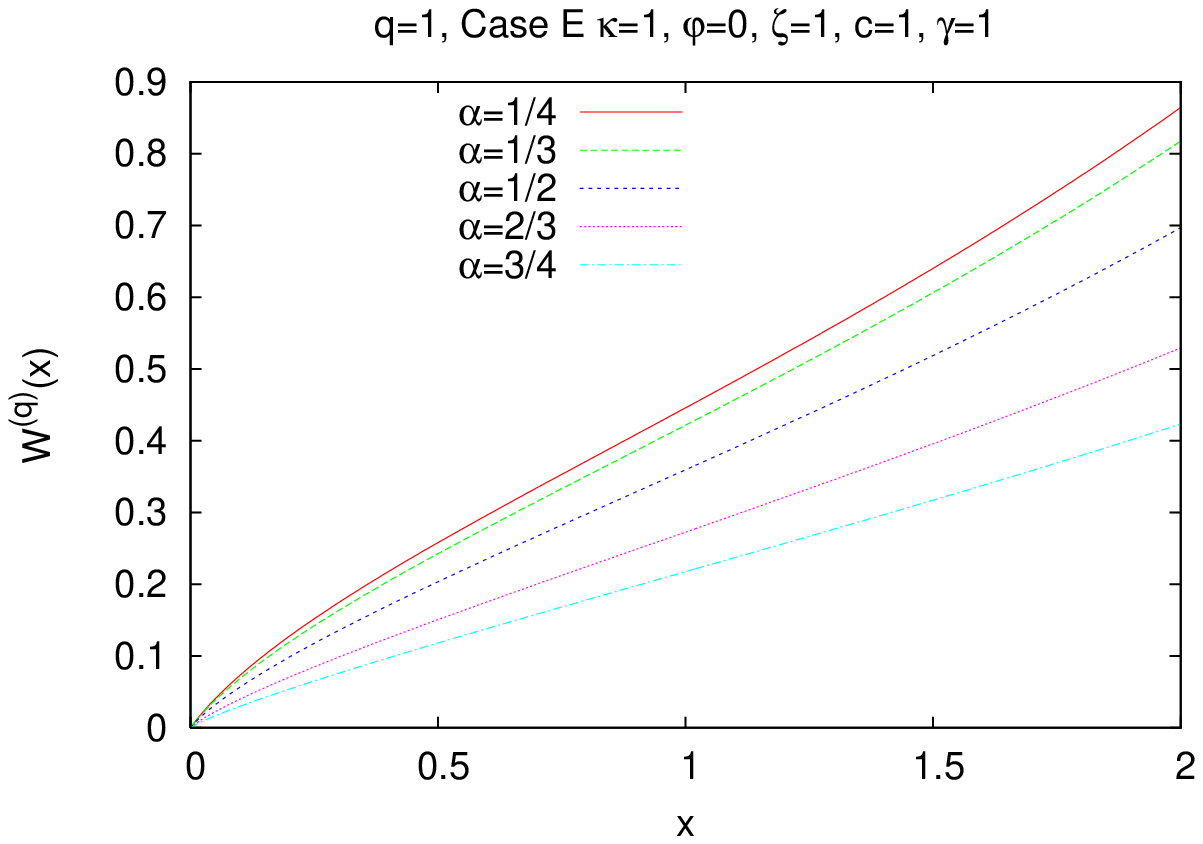}\\
\includegraphics[width=7cm,height=6cm]{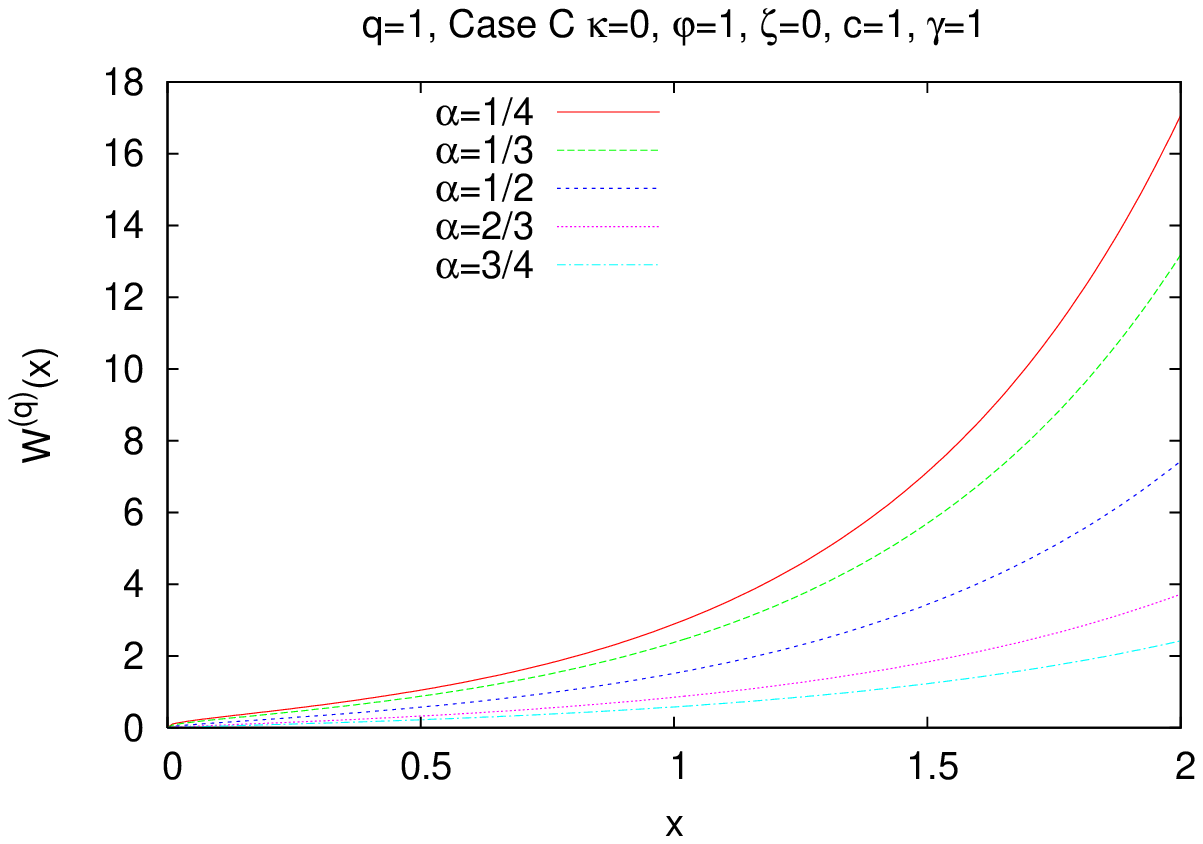}&\includegraphics[width=7cm,height=6cm]{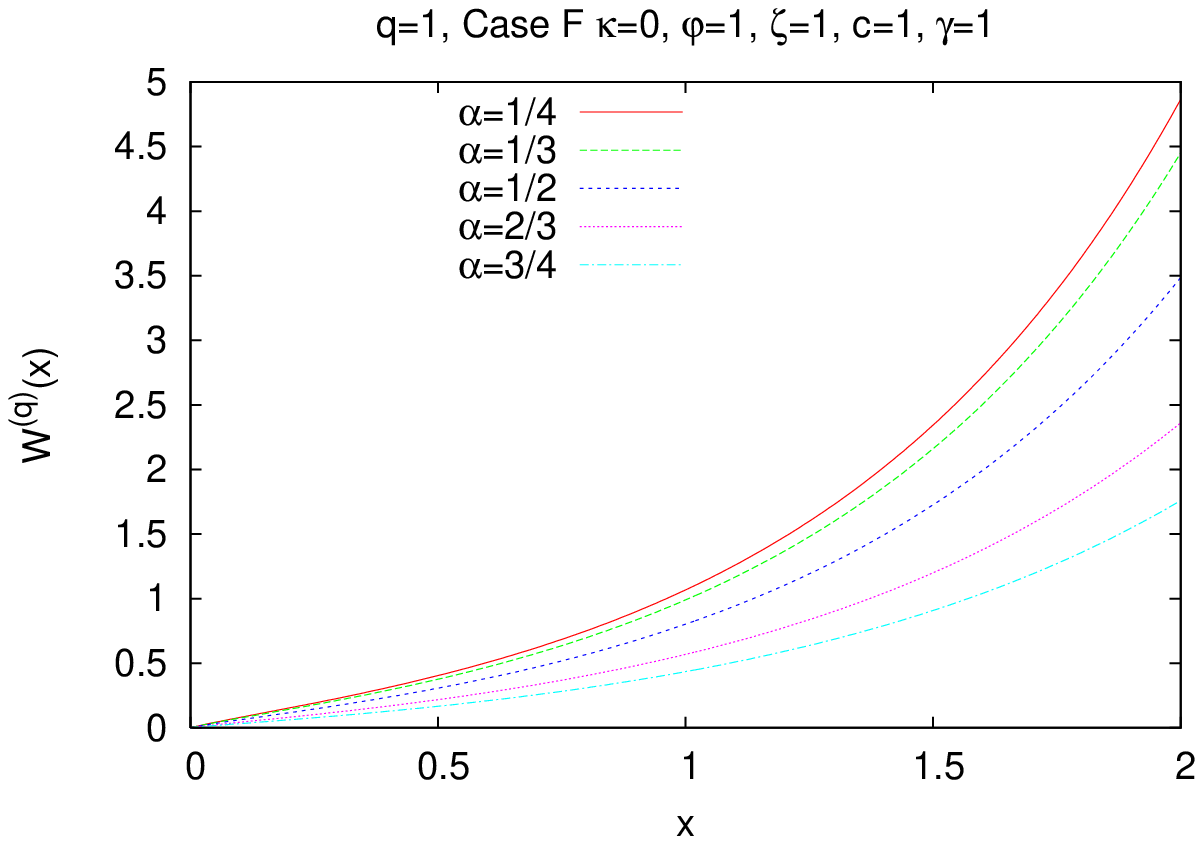}
\end{tabular}
\caption{Scale functions $W^{(q)}(x)$ for the GTSC class 
with $q=1$ and $0<\alpha<1$}
\label{fig2}
\end{figure}

\section{Additional Remarks}
\subsection{Extensions and ramifications}
Our methods apply also to the {\em integrated $q$-scale functions},
see \cite[Theorem~8.1, p.214]{Kyp2006}
\begin{equation}
Z^{(q)}(x)=1+q\int_0^xW^{(q)}(y)\,dy,
\end{equation}
for the Gaussian tempered stable convolution class with rational stability
parameter.

Also we could treat more complicated Gaussian tempered stable convolutions
with
\begin{equation}
\phi(\theta)=\kappa+\zeta\theta+\sum_{i=1}^nc_i\Gamma(-\alpha_i)
(\gamma_i^{\alpha_i}-(\gamma_i+\theta)^{\alpha_i})
\end{equation}
with $\alpha_1,\ldots,\alpha_n$ rational numbers in $[-1,1)$.

The key observation is simply, that any Laplace transform, that
is of the form $f(\theta^\alpha)$ with $f$ a rational function
and $\alpha$ some number,
can be inverted in terms of the partial fraction decomposition
of $f$ using Mittag-Leffler functions and its derivatives.

In other very recent work, \cite{KR2007} have looked at families of scale functions which are the result of choosing the descending ladder height process to have a Laplace exponent which is a special or complete Bernstein function.

\subsection{Numerical inversion of the Laplace transform}
While our formulas provide closed form expressions for rational
$\alpha$ in terms of partial fraction decompositions and Mittag-Leffler
functions, those expression can be rather lengthy,
involve the zeros of polynomials of high degree, 
and, experience shows them to be numerically problematic when $\alpha=m/n$ 
unless $n$ is quite small.

Alternatively one may always consider numerical inversion of
the Laplace transform of $W^{(q)}$, which is extremely simple
for the GTSC class. Moreover this approach works also
for irrational $\alpha$.
We use the Bromwich inversion integral for the Laplace transform,
\begin{equation}\label{Bromwich}
W^{(q)}(x)=\frac1{2\pi i}
\int\limits_{r-i\infty}^{r+i\infty}\frac{e^{\theta x}d\theta}{\psi(\theta)-q}
\end{equation}
where $r>\Phi(q)$.
Let us consider the integrand for $|\theta|\to\infty$ with $\Re(\theta)\geq r$.
The term $e^{\theta x}$ is bounded.
If $\zeta>0$ we have $\psi(\theta)\sim\zeta\theta^2$,
if $\zeta=0$ and $0<\alpha<1$ we have $\psi(\theta)\sim 
-c\Gamma(-\alpha)\theta^{\alpha+1}$ and (\ref{Bromwich}) is a proper
Lebesgue integral. Otherwise it must be interpreted as a principal
value integral. 
Some remarks how to handle this situation and how to improve the numerics
and a piece of Mathematica code, that evaluates the scale function by 
(straightforward, and numerically unsophisticated) integration, 
are in the extended report.
For further material on the numerical evaluation of scale function
see also \cite[Chapter7, p.111ff]{Surya-Thesis}.

\subsection{Examples of applications}
Let us briefly conclude by re-enforcing the genuine importance of establishing explicit examples of scale functions from the point of view of modelling  by giving three classical examples from the theory of applied probability. It will suffice to consider the role of $0$-scale functions. We shall do so in the light of some of the explicit examples above.

\begin{example}
Recent literature suggests that a more modern approach to the theory of ruin
should involve considering, in place of the classical Cram\'er-Lundberg process,
a spectrally negative L\'evy process which drifts to $\infty$ for the risk
process. As noted earlier, for any such given L\'evy process with Laplace
exponent $\psi$,  the quantity $1 - \psi'(0+)W(x)$ is the probability of ruin
when the L\'evy process is issued from $x$ at time $0$.

Consider the following parent process
\begin{equation}
X_t  = (\kappa+\lambda) t - S^{(1)}_t - S^{(2)}_t, \, t\geq 0
\end{equation}
where $\{S^{(i)}_t : t\geq 0\}$, $i=1,2$,  are subordinators 
whose \Levy{} measures, $\nu^{(i)}$, $i=1,2$, are identified by 
\begin{equation}
\nu^{(1)}(dx) = \lambda \frac{\gamma^{\nu+1}}{\Gamma(\nu)}x^{\nu - 1} e^{-\gamma x}dx \ \text{ and }  \ \nu^{(2)}(dx)= \lambda \frac{\gamma^\nu}{\Gamma(\nu)}(1-\nu) x^{\nu-2}e^{-\gamma x}dx
\end{equation}
for $x>0$.

Following the interpretation of L\'evy risk processes in \cite{KK2006} and \cite{Vondra} we may think of the parent process as the result of premiums collected at rate $\lambda+\kappa$ to offset the result of two competing claim processes. The first has claims arriving at a compound Poisson rate $\lambda$ which are gamma distributed and the second has an infinite intensity of small claims whose aggregate behaviour is similar to that of a stable subordinator with index $1-\nu$. Moreover, these claims are interlaced with larger claims of finite intensity, the tail of whose distribution are similar to those of gamma distributed jumps.

Recalling the discussion in Section \ref{TSparent} we have from the second part of Theorem \ref{niceform}  that the probability of ruin from an initial reserve $x>0$ is conveniently given by
\begin{equation}
1 -  \frac{1}{\lambda+\kappa} -\frac{\rho\gamma^\nu}{\lambda +\kappa}\int_0^x
y^{\nu-1}e^{-\gamma y} E_{\nu,\nu}(\rho\gamma^\nu y^\nu) d y
\end{equation}
\end{example}

\begin{example} 
If instead we consider the parent process in the previous example reflected in its supremum, then we are looking at the
workload of a so called $M/\Pi/1$ queue (cf \cite{DGM2004}). That is to say, a queue in which work is processed with a
constant rate, say $\mu$, and which arrives in packets according to a Poisson point process such that a job of size $x>0$
arrives in the interval $(t,t+dt)$ with probability $\Pi(dx)dt + o(dt)$ where $\int_0^\infty (1\wedge  x)\Pi(dx)<\infty$.
For the case at hand $\mu = \lambda+\kappa$ and $\Pi(dx)  = \nu^{(1)}(dx) +\nu^{(2)}(dx)$ and there is again the possibility
of interpreting the incoming work as the result of two competing mechanisms. In this case, the stationary distribution of
the workload is given by one minus the ruin probability given above.
\end{example}

\begin{example}
In the very recent work of \cite{Ronnie2007}, the classical dividend control problem of De Finetti's dividend problem has
been considered in the context of the driving risk process being a general spectrally negative L\'evy process denoted as
usual $X$. Here the objective is to find the optimal strategy and value to the following control problem
\begin{equation}
v(x) = \sup_{\pi}\mathbb{E}_x\left(\int_0^{\sigma^\pi} e^{-qt }dL^\pi_ t\right)
\end{equation}
where $q>0$, $\sigma^\pi = \inf\{ t>0 :  X_t - L^\pi_t <0\}$ and  $L^\pi =\{L^\pi_t : t\geq 0\}$ is the process of dividends paid out associated with the strategy $\pi$ and the supremum is taken over all strategies $\pi$ such  that $L^\pi$ is a non-decreasing, left-continuous adapted process which starts at zero and for which ruin cannot occur by a dividend pay-out.

In \cite{Ronnie2007}, the author proves the remarkable fact that, under the assumption that the dual of the underlying L\'evy process
has a completely monotone density, for each $q>0$, the associated $q$-scale function, $W^{(q)}$, has a first derivative
which is strictly convex on $(0,\infty)$. Moreover, the point $a^* : = \inf\{a \geq 0 : W^{(q)\prime} (a)\leq
W^{(q)\prime}(x) \text{ for all }x\geq 0\}$ is the level of the barrier which characterizes the optimal strategy. The latter
being that dividends are paid out in such a way that the aggregate process has the dynamics of the underlying L\'evy process
reflected at the barrier $a^*$. The value function of this strategy is given by 
 \begin{equation}
 v(x) = \left\{
 \begin{array}{ll}
\frac{W^{(q)} (x)}{W^{(q)\prime}(a)} & \text{ for }0\leq x\leq a \\
x-a + \frac{W^{(q)} (a)}{W^{(q)\prime}(a)} & \text{ for }x> a
 \end{array}
 \right.
 \end{equation}
If we are to take any of the parent processes mentioned in this paper as the underlying
spectrally negative L\'evy process, then it is immediately obvious from (\ref{cmconv})
that they all have the property that their dual has a L\'evy density which is
completely monotone. In such a case one may observe the value $a^*$ graphically (see for example Figure \ref{fig2}) as well as compute it numerically to a reasonable degree of exactness  with the help of software such as Mathematica. 
\end{example}
\section*{Acknowledgments} This work was initiated when both authors were attending the
mini-workshop {\it L\'evy processes and Related Topics in Modelling} in Oberwolfach.  We are
grateful to the organizers of this workshop and MFO for the opportunity it has provided us. We
would also like to thank Ronnie Loeffen for hinting that previously unnoticed scale functions
are to be found in the references of \cite{Furrer1998}, \cite{Asm2000} and \cite{Abate-Whitt}.
Thanks also go to Victor Rivero who commented on earlier versions of this paper.
\bibliography{newscale-initials}

\begin{thebibliography}{61}
\providecommand{\natexlab}[1]{#1}
\providecommand{\url}[1]{\texttt{#1}}
\expandafter\ifx\csname urlstyle\endcsname\relax
  \providecommand{\doi}[1]{doi: #1}\else
  \providecommand{\doi}{doi: \begingroup \urlstyle{rm}\Url}\fi

\bibitem[Abate and Whitt(1999)]{Abate-Whitt}
J.~Abate and W.~Whitt.
\newblock Explicit {$M/G/1$} waiting-time distributions for a class of
  long-taile service-time distributions.
\newblock \emph{Operations Research Letters}, 25:\penalty0 25--31, 1999.

\bibitem[Alili and Kyprianou(2005)]{AK2004}
L~Alili and A.E. Kyprianou.
\newblock Some remarks on first passage of {L\'evy} processes, the {American}
  put and pasting principles.
\newblock \emph{The Annals of Applied Probability}, 15\penalty0 (3):\penalty0
  2062--2080, 2005.

\bibitem[Asmussen(2000)]{Asm2000}
S.~Asmussen.
\newblock \emph{Ruin probabilities}.
\newblock World Scientific Publishing, 2000.

\bibitem[Avram et~al.(2002)Avram, Chan, and Usabel]{ACU2002}
F.~Avram, T.~Chan, and M.~Usabel.
\newblock On the valuation of constant barrier options under spectrally
  one-sided exponential {L\'evy} models and {Carr}'s approximation for
  {American} puts.
\newblock \emph{Stochastic Processes and their Applications}, 100:\penalty0
  75--107, 2002.

\bibitem[Avram et~al.(2004)Avram, Kyprianou, and Pistorius]{AKP2004}
F.~Avram, A.E. Kyprianou, and M.~R. Pistorius.
\newblock Exit problems for spectrally negative {L\'evy} processes and
  applications to ({C}anadized) {R}ussian options.
\newblock \emph{The Annals of Applied Probability}, 14\penalty0 (1):\penalty0
  215--238, 2004.

\bibitem[Avram et~al.(2007)Avram, Palmowski, and Pistorius]{APP2007}
F.~Avram, Z.~Palmowski, and M.~R. Pistorius.
\newblock On the optimal dividend problem for a spectrally negative {L\'evy}
  process.
\newblock \emph{Annals of Applied Probability}, 17:\penalty0 156--180, 2007.

\bibitem[Bekker et~al.(2007)Bekker, Boxma, and Kella]{BBK2007}
R.~Bekker, O.~Boxma, and O.~Kella.
\newblock Queues with delays in two-state strategies and {L\'evy} input.
\newblock {EURANDOM} report 2007--008, 2007.
\newblock Submitted for publication.

\bibitem[Bertoin(1996)]{Ber1996}
J.~Bertoin.
\newblock \emph{L\'evy processes}, volume 121 of \emph{Cambridge Tracts in
  Mathematics}.
\newblock Cambridge University Press, Cambridge, 1996.

\bibitem[Bertoin(1997)]{Ber1997}
J.~Bertoin.
\newblock Exponential decay and ergodicity of completely asymmetric {L\'evy}
  processes in a finite interval.
\newblock \emph{The Annals of Applied Probability}, 7\penalty0 (1):\penalty0
  156--169, 1997.

\bibitem[Bertoin et~al.(2004)Bertoin, Roynette, and Yor]{BRY}
J.~Bertoin, B.~Roynette, and M.~Yor.
\newblock Some connections between (sub)critical branching mechanisms and
  bernstein functions.
\newblock 2004.
\newblock \texttt{arXiv:math/0412322v1 [math.PR]}.

\bibitem[Bingham(1976)]{Bin1976}
N.~H. Bingham.
\newblock Continuous branching processes and spectral positivity.
\newblock \emph{Stochastic Processes and their Applications}, 4\penalty0
  (3):\penalty0 217--242, 1976.

\bibitem[Boxma and Cohen(1998)]{Boxma-Cohen}
O.~J. Boxma and J.W. Cohen.
\newblock The {M/G/1} queue with heavy-tailed service-time distribution.
\newblock \emph{IEEE Journal on Selected Areas in Communications}, 16:\penalty0
  749--763, 1998.

\bibitem[Chan and Kyprianou(2007)]{CK2007}
T.~Chan and A.E. Kyprianou.
\newblock Smoothness of scale functions for spectrally negative {L\'evy}
  processes.
\newblock Preprint, 2007.

\bibitem[Chaumont(1994)]{Chau94}
L.~Chaumont.
\newblock Sur certains processus de {L\'evy} conditionn\'es \`{a} rester
  positifs.
\newblock \emph{Stochastics and stochastics reports}, 47:\penalty0 1--20, 1994.

\bibitem[Chaumont(1996)]{Chau96}
L.~Chaumont.
\newblock Conditionings and path decompositions for {L\'evy} processes.
\newblock \emph{Stochastics and stochastics reports}, 64:\penalty0 39--54,
  1996.

\bibitem[Chaumont and Caballero(2006)]{CC06}
L.~Chaumont and M.E. Caballero.
\newblock Conditioned stable {L\'evy} processes and the {Lamperti}
  representation.
\newblock \emph{Journal of Applied Probability}, 43:\penalty0 967--983, 2006.

\bibitem[Chaumont et~al.(2007)Chaumont, Kyprianou, and Pardo]{CKP}
L.~Chaumont, A.~E. Kyprianou, and J.C. Pardo.
\newblock {Wiener}-{Hopf} factorization and some explicit identities associated
  with positive self-similar {Markov} processes.
\newblock \emph{Preprint}, 2007.

\bibitem[Chiu and Yin(2005)]{CY2005}
S.N. Chiu and C.~Yin.
\newblock Passage times for a spectrally negative {L\'evy} process with
  applications to risk theory.
\newblock \emph{Bernoulli}, 11\penalty0 (3):\penalty0 511--522, 2005.

\bibitem[Cont and Tankov(2003)]{cont}
R.~Cont and P.~Tankov.
\newblock \emph{Financial modelling with Jump Processes}.
\newblock Chapman \& Hall, CRC Press., 2003.

\bibitem[Doney and Kyprianou(2006)]{DK2006}
A.D. Doney and A.E. Kyprianou.
\newblock Overshoots and undershoots of {L\'evy} processes.
\newblock \emph{The Annals of Applied Probability}, 16\penalty0 (1):\penalty0
  91--106, 2006.

\bibitem[Doney(2005{\natexlab{a}})]{Don2005}
R.~A. Doney.
\newblock Some excursion calculations for spectrally one-sided {L\'evy}
  processes.
\newblock In \emph{S\'eminaire de Probabilit\'es XXXVIII}, volume 1857 of
  \emph{Lecture Notes in Math.}, pages 5--15. Springer, Berlin,
  2005{\natexlab{a}}.

\bibitem[Doney(1991)]{Don1991}
R.A. Doney.
\newblock Hitting probabilities for spectrally positive {L\'evy} processes.
\newblock \emph{Journal of the London Mathematical Society. Second Series},
  44\penalty0 (3):\penalty0 566--576, 1991.

\bibitem[Doney(2007)]{Don2007}
R.A. Doney.
\newblock Fluctuation theory for {L\'evy} processes.
\newblock In \emph{\'Ecole d'\'et\'e de probabilit\'es de Saint-Flour,
  XXXV---2005}, volume 1897 of \emph{Lecture Notes in Math.} Springer, Berlin,
  2007.

\bibitem[Doney(2005{\natexlab{b}})]{doneySdP}
R.A. Doney.
\newblock Some excursion calculations for spectrally one-sided {L\'evy}
  processes.
\newblock \emph{{S\'eminaire} de {Probabilit\'es}}, XXXVIII:\penalty0 5--15,
  2005{\natexlab{b}}.

\bibitem[Dube et~al.(2004)Dube, Guillemin, and Mazumdar]{DGM2004}
P.~Dube, F.~Guillemin, and R.~R. Mazumdar.
\newblock Scale functions of {L\'evy} processes and busy periods of
  finite-capacity {$M/GI/1$} queues.
\newblock \emph{Journal of Applied Probability}, 41\penalty0 (4):\penalty0
  1145--1156, 2004.

\bibitem[Emery(1973)]{Eme1973}
D.~J. Emery.
\newblock Exit problem for a spectrally positive process.
\newblock \emph{Advances in Applied Probability}, 5:\penalty0 498--520, 1973.
\newblock ISSN 0001-8678.

\bibitem[Frans{\'e}n and Wrigge(1984)]{FW1984}
A.~Frans{\'e}n and S.~Wrigge.
\newblock Calculation of the moments and the moment generating function for the
  reciprocal gamma distribution.
\newblock \emph{Mathematics of Computation}, 42\penalty0 (166):\penalty0
  601--616, 1984.

\bibitem[Furrer(1998)]{Furrer1998}
H.~Furrer.
\newblock Risk processes perturbed by $\alpha$-stable {L\'evy} motion.
\newblock \emph{Scandinavian. Actuarial Journal}, 1:\penalty0 59--74, 1998.

\bibitem[Hilberink and Rogers(2002)]{HR2002}
B.~Hilberink and L.~C.~G. Rogers.
\newblock Optimal capital structure and endogenous default.
\newblock \emph{Finance and Stochastics}, 6\penalty0 (2):\penalty0 237--263,
  2002.

\bibitem[Huzak et~al.(2004)Huzak, Perman, \v{S}iki\'c, and
  Vondra\v{c}ek]{Vondra}
M.~Huzak, M.~Perman, H.~\v{S}iki\'c, and Z.~Vondra\v{c}ek.
\newblock Ruin probabilities for competing claim processes.
\newblock \emph{Journal of Applied Probability}, 41:\penalty0 679--690, 2004.

\bibitem[Jacob(2005)]{JacIII}
N.~Jacob.
\newblock \emph{Pseudo differential operators and {Markov} processes.
  {Vol.III}}.
\newblock Imperial College Press, London, 2005.
\newblock Markov processes and applications.

\bibitem[Kl\"uppelberg and Kyprianou(2006)]{KK2006}
C.~Kl\"uppelberg and A.E. Kyprianou.
\newblock On extreme ruinous behaviour of {L\'evy} insurance risk processes,
  2006.

\bibitem[Kl{\"u}ppelberg et~al.(2004)Kl{\"u}ppelberg, Kyprianou, and
  Maller]{KKM2004}
C.~Kl{\"u}ppelberg, A.~E. Kyprianou, and R.~A. Maller.
\newblock Ruin probabilities and overshoots for general {L\'evy} insurance risk
  processes.
\newblock \emph{The Annals of Applied Probability}, 14\penalty0 (4):\penalty0
  1766--1801, 2004.

\bibitem[Korolyuk(1974)]{Kor1974}
V.~S. Korolyuk.
\newblock {Boundary problems for a compound Poisson process.}
\newblock \emph{Theory Probab. Appl.}, 19:\penalty0 1--14, 1974.

\bibitem[Korolyuk(1975)]{Kor1975}
V.~S. Korolyuk.
\newblock {On ruin problems for a compound Poisson process.}
\newblock \emph{Theory Probab. Appl.}, 20:\penalty0 374--376, 1975.

\bibitem[Kou and Wang(2003)]{KouWang}
S.~Kou and H.~Wang.
\newblock First passage times of a jump diffusion process.
\newblock \emph{Advances of Applied Probability}, 35:\penalty0 504--531, 2003.

\bibitem[Krell(2007)]{krell2007}
N.~Krell.
\newblock Multifractal spectra and precise rates of decay in homogeneous
  fragmentations.
\newblock \emph{Preprint}, 2007.

\bibitem[Kyprianou(2006)]{Kyp2006}
A.~E. Kyprianou.
\newblock \emph{Introductory lectures on fluctuations of {L\'evy} processes
  with applications}.
\newblock Universitext. Springer-Verlag, Berlin, 2006.

\bibitem[Kyprianou and Palmowski(2006)]{KP2006}
A.E. Kyprianou and Z.~Palmowski.
\newblock Quasi-stationary distributions for {L\'evy} processes.
\newblock \emph{Bernoulli}, 12\penalty0 (4):\penalty0 571--581, 2006.

\bibitem[Kyprianou and Palmowski(2007)]{KP2007}
A.E. Kyprianou and Z.~Palmowski.
\newblock Distributional study of de {Finetti's} dividend problem for a general
  {L\'evy} insurance risk process.
\newblock \emph{Journal of Applied Probability}, 44\penalty0 (2):\penalty0
  428--443, 2007.

\bibitem[Kyprianou and Rivero(2007)]{KR2007}
A.E. Kyprianou and V.~Rivero.
\newblock Special, conjugate and complete scale functions for spectrally
  negative {L\'evy} processes.
\newblock \emph{Preprint}, 2007.

\bibitem[Kyprianou and Surya(2007)]{KS2007}
A.E. Kyprianou and B.~A. Surya.
\newblock Principles of smooth and continuous fit in the determination of
  endogenous bankruptcy levels.
\newblock \emph{Finance and Stochastics}, 11\penalty0 (1):\penalty0 131--152,
  2007.

\bibitem[Kyprianou et~al.(2008)Kyprianou, Rivero, and Song]{KRS2008}
A.E. Kyprianou, V.~Rivero, and R.~Song.
\newblock Smoothness and convexity of scale functions with applications to de
  {Finetti's} control problem.
\newblock \emph{Preprint}, 2008.

\bibitem[Lambert(2000)]{Lam2000}
A.~Lambert.
\newblock Completely asymmetric {L\'evy} processes confined in a finite
  interval.
\newblock \emph{Annales de l'Institut Henri Poincar\'e. Probabilit\'es et
  Statistiques}, 36\penalty0 (2):\penalty0 251--274, 2000.

\bibitem[Lambert(2007)]{Lam2007}
A.~Lambert.
\newblock Quasi-stationary distributions and the continuous-state branching
  process conditioned to be never extinct.
\newblock \emph{Electronic Journal of Probability}, 12:\penalty0 420--446,
  2007.

\bibitem[Loeffen(2007)]{Ronnie2007}
R.~Loeffen.
\newblock On optimality of the barrier strategy in de {Finetti's} dividend
  problem for spectrally negative l\'evy processes.
\newblock \emph{Preprint}, 2007.

\bibitem[Mordecki and Lewis(2005)]{Mor}
E.~Mordecki and A.~Lewis.
\newblock {Wiener}-{Hopf} factorization for {L\'evy} processes having negative
  jumps with rational transforms.
\newblock \emph{Submitted}, 2005.

\bibitem[Pistorius(2003)]{Pis2003}
M.~R. Pistorius.
\newblock On doubly reflected completely asymmetric {L\'evy} processes.
\newblock \emph{Stochastic Processes and their Applications}, 107\penalty0
  (1):\penalty0 131--143, 2003.

\bibitem[Pistorius(2004)]{Pis2004}
M.~R. Pistorius.
\newblock On exit and ergodicity of the spectrally one-sided {L\'evy} process
  reflected at its infimum.
\newblock \emph{Journal of Theoretical Probability}, 17\penalty0 (1):\penalty0
  183--220, 2004.

\bibitem[Pistorius(2005)]{Pis2005}
M.~R. Pistorius.
\newblock A potential-theoretical review of some exit problems of spectrally
  negative {L\'evy} processes.
\newblock In \emph{S\'eminaire de Probabilit\'es XXXVIII}, volume 1857 of
  \emph{Lecture Notes in Math.}, pages 30--41. Springer, Berlin, 2005.

\bibitem[Pistorius(2007)]{Pis2006}
M.~R. Pistorius.
\newblock An excursion theoretical approach to some boundary crossing problems
  and the {Skorokhod} embedding for reflected {L\'evy} processes.
\newblock \emph{{S\'eminaire} de {Probabilit\'es}}, 40:\penalty0 287--308,
  2007.

\bibitem[Renaud and Zhou(2007)]{RZ2007}
J-F. Renaud and X.~Zhou.
\newblock Distribution of the dividend payments in a general {L\'evy} risk
  model.
\newblock \emph{Journal of Applied Probability}, 44\penalty0 (2):\penalty0
  420--427, 2007.

\bibitem[Rogers(1990)]{Rog1990}
L.~C.~G. Rogers.
\newblock The two-sided exit problem for spectrally positive {L\'evy}
  processes.
\newblock \emph{Advances in Applied Probability}, 22\penalty0 (2):\penalty0
  486--487, 1990.

\bibitem[Schoutens(2003)]{schoutens}
W.~Schoutens.
\newblock \emph{L\'evy processes in finance.}
\newblock Probability and Statistics. Wiley, 2003.

\bibitem[Steutel and van Harn(2004)]{steutelvanharn}
F.~W. Steutel and K.~van Harn.
\newblock \emph{Infinite divisibility of probability distributions on the real
  line}, volume 259 of \emph{Pure and Applied Mathematics}.
\newblock Marcel Dekker, 2004.

\bibitem[Suprun(1976)]{Sup1976}
V.~N. Suprun.
\newblock The ruin problem and the resolvent of a killed independent increment
  process.
\newblock \emph{Akademiya Nauk Ukrainsko\u\i\ SSR. Institut Matematiki.
  Ukrainski\u\i\ Matematicheski\u\i\ Zhurnal}, 28\penalty0 (1):\penalty0
  53--61, 142, 1976.

\bibitem[Surya(2007)]{Surya-Thesis}
B.~A. Surya.
\newblock \emph{Optimal stopping problems driven by {L\'evy} processes and
  Pasting Principles}.
\newblock Proefschrift, Utrecht University, 2007.
\newblock submitted to Journal of Applied Probability.

\bibitem[Tak{\'a}cs(1966)]{Tak1966}
L.~Tak{\'a}cs.
\newblock \emph{Combinatorial methods in the theory of stochastic processes}.
\newblock John Wiley \& Sons Inc., New York, 1966.

\bibitem[Vigon(2002{\natexlab{a}})]{Vig2002}
V.~Vigon.
\newblock Votre {L\'evy} rampe-t-il?
\newblock \emph{Journal of the London Mathematical Society. Second Series},
  65\penalty0 (1):\penalty0 243--256, 2002{\natexlab{a}}.

\bibitem[Vigon(2002{\natexlab{b}})]{Vigonthesis}
V.~Vigon.
\newblock \emph{Simplifiez vos {L\'evy} en titillant la factorisation de
  Wiener-Hopf}.
\newblock Laboratoire de Math\'ematiques de L'INSA de Rouen.,
  2002{\natexlab{b}}.

\bibitem[Zolotarev(1964)]{Zol1964}
V.~M. Zolotarev.
\newblock The moment of first passage of a level and the behaviour at infinity
  of a class of processes with independent increments.
\newblock \emph{Akademija Nauk SSSR. Teorija Verojatnoste\u\i\ i ee
  Primenenija}, 9:\penalty0 724--733, 1964.

\end{thebibliography}

\end{document}